\documentclass[a4paper,11pt]{article}
\title{Stable sets in \{ISK4,wheel\}-free graphs}

\author{Martin Milani\v c\thanks{University of Primorska, UP IAM, UP FAMNIT, Koper, Slovenia.
Partially supported by the Slovenian Research Agency (I$0$-$0035$, research program P$1$-$0285$ and research projects N$1$-$0032$, J$1$-$5433$, J$1$-$6720$, J$1$-$6743$, and J$1$-$7051$). E-mail: \texttt{martin.milanic@upr.si}.}
\and Irena Penev \thanks{Department of Applied Mathematics and Computer Science, Technical University of Denmark, Lyngby, Denmark. Most of this work was conducted while the author was at Universit\'e de Lyon, LIP, ENS de Lyon, Lyon, France. Partially
supported by the ANR project \textsc{Stint} under \textsc{Contract ANR-13-BS02-0007}, by the LABEX MILYON (ANR-10-LABX-0070) of Universit\'e de
    Lyon, within the program ‘‘Investissements d'Avenir’’
    (ANR-11-IDEX-0007) operated by the French National Research Agency
    (ANR), and by the ERC Advanced Grant GRACOL, project number 320812. Email: \texttt{ipen@dtu.dk}.}
  \and Nicolas Trotignon\thanks{Unviersit\'e de Lyon, CNRS, LIP, ENS de
  Lyon. Partially supported by ANR project Stint under reference
  ANR-13-BS02-0007 and by the LABEX MILYON (ANR-10-LABX-0070) of
  Universit\'e de Lyon, within the program ‘‘Investissements
  d'Avenir’’ (ANR-11-IDEX-0007) operated by the French National
  Research Agency (ANR). E-mail: \texttt{nicolas.trotignon@ens-lyon.fr}.}
}

\usepackage{amsthm}
\usepackage{color}
\usepackage{amssymb}
\usepackage{amsbsy}
\usepackage{mathrsfs}
\usepackage{amstext}
\usepackage{stmaryrd}
\newtheorem{lemma}{Lemma}[section]
\newtheorem{proposition}[lemma]{Proposition}
\newtheorem{theorem}[lemma]{Theorem}
\newtheorem{corollary}[lemma]{Corollary}

\begin{document}
\maketitle

\begin{abstract}
 An \emph{ISK4} in a graph $G$ is an induced subgraph of $G$ that is
 isomorphic to a subdivision of $K_4$ (the complete graph on four
 vertices). A \emph{wheel} is a graph that consists of a chordless cycle,
 together with a vertex that has at least three
 neighbors in the cycle. A graph is \{ISK4,wheel\}-free if it has
 no ISK4 and does not contain a wheel as an induced subgraph. We
 give an $O(|V(G)|^7)$-time algorithm to compute the maximum weight of a
 stable set in an input weighted \hbox{\{ISK4,wheel\}-free} graph $G$ with
 non-negative integer weights.
\end{abstract}

\section{Introduction}
\label{s:intro}

All graphs in this paper are finite and simple.
An \emph{ISK4} in a graph $G$ is an induced subgraph of $G$ that is
isomorphic to a subdivision of $K_4$ (the complete graph on four
vertices). An ISK4-free graph is a graph $G$ that contains no ISK4 
(that is, no induced subgraph of $G$ is isomorphic to a subdivision of $K_4$).
The class of ISK4-free graphs contains all series-parallel graphs, and also all
line graphs of graphs of maximum degree at most three.

L{\'e}v{\^e}que, Maffray, and Trotignon~\cite{MR2927414} proved a
decomposition theorem for ISK4-free graphs, but gave no
algorithmic applications. In particular, no polynomial-time
algorithms and no hardness proofs are known for the following
problems in the class of ISK4-free graphs: recognition, maximum
stable set, and coloring. (Finding a maximum clique in an
ISK4-free graph is of course trivial because every clique in such
a graph is of size at most three.)

A \emph{wheel} is a graph that consists of a chordless cycle,
together with a vertex (called the {\em center} of the wheel) that has at least three neighbors in
the cycle. A graph is \emph{wheel-free} if none of its induced subgraphs
is a wheel. Wheel-free graphs have a number of structural properties (see for
instance~\cite{abChTrVu:moplex,aboulker.c.s.T:wfpg,diotRaTrVu:15}).
However, the maximum stable set problem is easily seen to remain 
NP-hard even when restricted to the class of wheel-free graphs. To see this, denote by
$\alpha(G)$ the {\em stability number} (i.e., the maximum size of a stable set)
of a graph $G$, and consider the operation of subdividing every edge of
$G$ twice. This yields a graph $G'$ that is wheel-free (because every vertex
of degree at least three in $G'$ has only neighbors of degree two, and so it
cannot be the center of a wheel). As observed by
Poljak~\cite{poljak:74}, $\alpha(G') = \alpha(G) + |E(G)|$, and so
computing the stability number of a wheel-free graph is as hard as computing it
in a general graph.

A graph is \emph{\{ISK4,wheel\}-free} if it is ISK4-free and
wheel-free. In~\cite{MR2927414}, a decomposition theorem is given for
\{ISK4,wheel\}-free graphs. (This theorem was obtained as a corollary of the
decomposition theorem for ISK4-free graphs from~\cite{MR2927414}.) The theorem
for \{ISK4,wheel\}-free graphs is stronger than the one for
ISK4-free graphs in the sense that the former theorem can be used to solve the recognition
and the coloring problems for \{ISK4,wheel\}-free graphs in
polynomial time. However, no other algorithmic application has
previously been reported.

In this paper, we investigate the {\em maximum weight stable set problem}
restricted to \{ISK4,wheel\}-free graphs. Let us be precise. First, by
a {\em weighted graph}, we mean an ordered pair $(G,w)$, where $G$ is a
graph and $w$ is a function (called a {\em weight function} for $G$) that
assigns to each vertex $v$ of $G$ a non-negative integer weight $w(v)$. The
{\em weight} of a set of vertices is the sum of the weights of its elements.
The {\em stability number} of a weighted graph $(G,w)$, denoted by
$\alpha(G,w)$, is the maximum weight of a stable set of $G$, and a maximum
weighted stable set of $(G,w)$ is a stable set whose weight is precisely
$\alpha(G,w)$. (If $(G,w)$ is a weighted graph and $H$ is an induced subgraph of $G$, then we will also write
$\alpha(H,w)$ for the stability number of the weighted graph $(H,w')$ where $w'$ is the restriction of $w$ to $V(H)$.)
The {\em maximum weight stable set problem} for a given class $\mathcal{G}$ of graphs is the problem of finding a maximum weight stable set
in a given weighted graph $(G,w)$ such that $G \in \mathcal{G}$. 
A {\em hereditary class} is a class of graphs that is closed under isomorphism and induced subgraphs (clearly,
the class of \{ISK4,wheel\}-free graphs is a hereditary class). The following is
well-known.

\begin{proposition}[folklore]\label{prop:number-to-set} Let $\mathcal{G}$ be a hereditary class. Suppose that $\mathcal{A}$ is an algorithm that computes the stability number of any weighted graph $(G,w)$ such that $G \in \mathcal{G}$ in $O(|V(G)|^k)$ time. Then there is an algorithm $\mathcal{B}$ that computes a maximum weight stable set of any graph $(G,w)$ such that $G \in \mathcal{G}$ in $O(|V(G)|^{\max\{k+1,3\}})$ time.
\end{proposition}

\begin{proof}
Let $(G,w)$ be an input graph with $G \in \mathcal{G}$, and set $n = |V(G)|$. If $G$ is the null graph, then the algorithm returns $\emptyset$ and stops. Otherwise, we choose a vertex $v \in V(G)$, we compute the graph $G \smallsetminus N[v]$ in $O(n^2)$ time (where $N[v]$ is the set consisting of $v$ and all its neighbors in $G$), and using the algorithm $\mathcal{A}$, we compute $\alpha(G,w)$ and $\alpha(G \smallsetminus N[v],w)$ in $O(n^k)$ time. Clearly, $w(v)+\alpha(G \smallsetminus N[v],w) \leq \alpha(G,w)$. If $w(v)+\alpha(G \smallsetminus N[v],w) = \alpha(G,w)$, then we recursively compute a maximum weight stable set $S$ of $(G \smallsetminus N[v],w)$, and the algorithm returns $\{v\} \cup S$ and stops. On the other hand, if $w(v)+\alpha(G \smallsetminus N[v],w) < \alpha(G,w)$, then we see that no maximum weight stable set of $(G,w)$ contains $v$. In this case, we compute $G \smallsetminus v$ in $O(n^2)$ time, we recursively compute a maximum weight stable set $S$ of $(G \smallsetminus v,w)$, and the algorithm returns $S$ and stops.

It is clear that the algorithm is correct. We make $O(n)$ recursive calls to the algorithm, and it follows that the total running time of the algorithm is $O(n^{\max\{k+1,3\}})$.
\end{proof}

In view of Proposition~\ref{prop:number-to-set}, from now on, we focus on constructing a polynomial-time
algorithm that computes the stability number of weighted \{ISK4,wheel\}-free graphs.

The decomposition theorem for \{ISK4,wheel\}-free graphs from~\cite{MR2927414} states
(roughly) that every such graph is either ``basic'' or admits a ``decomposition''
(that is, a way to break it up into smaller pieces). The basic classes are all fairly
easy to handle and the main difficulty is posed by the decompositions. One of the
decompositions, namely the clique-cutset (that is, a clique whose deletion yields
a disconnected graph), is easy to handle, but the other one is not: the ``proper
2-cutset.'' A {\em proper 2-cutset} of a graph $G$ is a pair of non-adjacent
vertices, say $a$ and $b$, such that $V(G) \smallsetminus \{a,b\}$ can be
partitioned into two non-empty sets $X$ and $Y$ so that there is no edge between
$X$ and $Y$, and neither $G[X \cup \{a,b\}]$ nor $G[Y \cup \{a,b\}]$ is a path
between $a$ and $b$.

 The problem with a proper 2-cutset $\{a, b\}$ of a graph $G$ is that
 a maximum stable set of $G$ may contain $a$ (but not $b$), or $b$ (but not $a$), or neither $a$ nor $b$,
 or both $a$ and $b$. This phenomenon also occurs in any induced
 subgraph of $G$ that contains $a$ and $b$, and in particular in any
 reasonable subgraph built for the purposes of a recursive algorithm.
 So, any naive attempt to build an algorithm for the maximum stable
 set problem relying on proper 2-cutsets should lead one to consider an
 exponential number of cases. In fact, the situation is even worse
 for proper 2-cutsets as shown by a hardness result that we explain
 now. Let $G$ be a graph. A \emph{2-extension} of $G$ is any graph
 obtained from $G$ by first deleting a vertex $v$ of degree two, with non-adjacent neighbors $a$ and $b$, 
 then adding four vertices of degree two forming a path $a-x_1-x_2-x_3-x_4-b$, and finally adding a vertex
 $x$ adjacent to $x_1$, $x_2$, $x_3$ and $x_4$. An \emph{extended
  bipartite} graph is any graph obtained from a bipartite graph by
 repeatedly applying 2-extensions. Trivially, extended bipartite
 graphs have a decomposition theorem: if $G$ is an extended bipartite graph, then either
 $G$ is bipartite, or $G$ was obtained from an even cycle by performing exactly one 2-extension, or $G$ has a proper 2-cutset.
 It is well-known that one can find the stability number of a (weighted) bipartite graph in polynomial time (see for instance~\cite{FaigleFrahling}). It is also clear that the stability number of a (weighted) graph obtained from an even cycle by performing exactly one 2-extension can be found in polynomial time (indeed, if $G$ is obtained from an even cycle by performing exactly one 2-extension, and $x,x_1,x_2,x_3,x_4$ are as in the definition of a 2-extension, then every stable set of $G$ is also a stable set of at least one of $G \smallsetminus \{x_2,x_3\}$, $G \smallsetminus \{x,x_1\}$, and $G \smallsetminus \{x,x_2\}$, and each of these three induced subgraphs of $G$ is either a path or an even cycle, and is therefore bipartite). 
 So, if proper 2-cutsets were a good tool for solving the maximum stable set problem,
 there should be a polynomial-time algorithm for solving this problem in
 extended bipartite graphs. However, it was shown in~\cite{nicolas.kristina:2-join} that the maximum stable
 set problem is NP-hard when restricted to extended bipartite graphs.
 The result is stated differently in~\cite{nicolas.kristina:2-join}, and so we reproduce it here
 for the sake of completeness.

\begin{proposition}
 The problem of computing the stability number of an input extended bipartite graph is NP-hard.
\end{proposition}

\begin{proof}
 Suppose there is a polynomial-time algorithm $\cal A$ for our
 problem. We prove the theorem by using $\cal A$ as a subroutine to
 solve the problem of computing the stability number of a general
 graph $G$ in polynomial time. First build $B$ by subdividing every
 edge of $G$ once. Let $X$ be the set of vertices of degree two in
 $B$ that arise from the subdivisions. Note that $B$ is bipartite,
 and $(X, V(B) \smallsetminus X)$ is a bipartition of $B$. Now build
 a graph $H$ from $B$ by applying a 2-extension to every vertex of
 $X$. By construction, $H$ is an extended bipartite graph, and it is
 easy to check that $\alpha(H) = \alpha(G) + 2 |E(G)|$. Thus, $\cal A$
 indeed allows one to compute the stability number of a general graph
 in polynomial time.
\end{proof}

Despite this negative result, we can use (a variant of) a proper 2-cutset
in the special case of \{ISK4,wheel\}-free graphs, mainly because the
basic classes are very restricted. We rely on what is called a
``trigraph,'' which is a graph where some edges are left ``undecided''
(the notion is from \cite{chudnovsky:these,chudnovsky:trigraphs}, and formal definitions are given in Section~\ref{s:trigraphs}).
The idea is as follows. When $G$ is a graph, $a$ and $b$ are non-adjacent vertices of $G$ whose deletion yields a disconnected graph, and $V(G) \smallsetminus \{a,b\}$ is partitioned into non-empty sets $X$ and $Y$ such that there are no edges between $X$ and $Y$, we build a (tri)graph on the vertex set $X \cup \{a,b\}$by keeping the edges of $G$, and by leaving the adjacency between $a$ and $b$ undecided. We
give weights to $a$ and $b$, and we also give weights specific to the pair
$\{a,b\}$. Roughly speaking, the weights associated with the vertices $a$ and $b$ and the pair $\{a,b\}$ ``encode'' the maximum weight of a stable set in graphs $G[Y]$, $G[Y \cup \{a\}]$, $G[Y \cup \{b\}]$, and $G[Y \cup \{a,b\}]$. We therefore need weights,
undecided adjacencies, and a way to handle the notion of the weight of
a stable set in this context. All this is captured in the notion of
weighted trigraph (we postpone the formal definition to Section~\ref{s:swi}).
We remark that a similar idea was previously used in~\cite{bulls-stable} in the context of bull-free graphs. However, the definition of a weighted trigraph was simpler in~\cite{bulls-stable} than in the present paper, as was the definition of the weight of a stable set in a weighted trigraph. The reason for this is that the decompositions that appear in the context of bull-free graphs are more convenient than proper 2-cutsets for the purposes of computing the stability number.

We complete the introduction by giving an outline of the paper.
In Section~\ref{s:trigraphs}, we define trigraphs and introduce some basic trigraph-related
terminology that we need. In Section~\ref{s:swi}, we define weighted trigraphs, explain
how to compute the weight of a set of vertices in a weighted trigraph, and prove several
results about weighted stable sets in weighted trigraphs. These properties are more complicated
than one might expect because of the weights associated with the undecided adjacencies. Because of
these weights, the weight of a set is not a monotone function (one could
increase the weight of a set by taking a subset). For this reason, all proofs need
to be written very carefully. In Section~\ref{s:dec}, we state a decomposition theorem for
\{ISK4,wheel\}-free trigraphs (see Theorem~\ref{thm-decomp}). 
Since the proof of this theorem is very similar to that of 
the decomposition theorem for ISK4-free graphs from~\cite{MR2927414},
we omit the proof of Theorem~\ref{thm-decomp} in the present paper. 
The interested reader can find a detailed proof in~\cite{ISK4wheel-decomp}. 
Interestingly, the fact
that our theorem concerns trigraphs rather than graphs does not substantially complicate the proof, even though
our theorem is formally stronger than the corresponding one for graphs. On the other hand, the fact that we
restrict our attention to the wheel-free case significantly simplifies our proof. We complete Section~\ref{s:dec}
by using Theorem~\ref{thm-decomp} to prove an ``extreme'' decomposition theorem for \{ISK4,wheel\}-free trigraphs
(see Theorem~\ref{alg-extreme-decomp} and Corollary~\ref{cor-extreme-decomp}). 
Roughly speaking, our extreme decomposition theorem states 
that every \{ISK4,wheel\}-free trigraph is either basic or admits a decomposition such that one
block of decomposition is basic. In Section~\ref{s:op}, we give a transformation from a weighted trigraph to a weighted 
graph that preserves the stability number. In Section~\ref{s:comp}, we use this transformation to compute the stability number
in our basic trigraphs. Again, the proofs have to be done carefully, because as proven at the very end of the
paper (see Theorem~\ref{Bip-NP}), it is NP-hard to compute the stability number of weighted bipartite
trigraphs, and so one should be suspicious of ``simple'' classes of trigraphs in our context. In Section~\ref{s:compClass},
we prove our main technical result: there is an $O(|V(G)|^7)$-time algorithm to compute the stability number of an
input weighted \{ISK4,wheel\}-free trigraph $(G,w)$ (see Theorem~\ref{thm:main}).
Since every weighted \{ISK4,wheel\}-free graph can be seen as a weighted \{ISK4,wheel\}-free trigraph, the algorithm from
Theorem~\ref{thm:main} can also be applied to \{ISK4,wheel\}-free graphs. Together with Proposition~\ref{prop:number-to-set}, this
yields an $O(|V(G)|^8)$-time algorithm that finds a maximum weight stable set of an input weighted \{ISK4,wheel\}-free graph
(see Corollary~\ref{thm:main-graphs}). In Section~\ref{s:bip}, we prove the above-mentioned Theorem~\ref{Bip-NP}, which states
that it is NP-hard to compute the stability number of a weighted bipartite trigraph.

\section{Trigraphs}
\label{s:trigraphs}
Given a set $S$, we denote by ${S \choose 2}$ the set of all subsets of $S$ of size two. A {\em trigraph} is an ordered pair $G = (V(G),\theta_G)$, where $V(G)$ is a finite set, called the {\em vertex set} of $G$ (members of $V(G)$ are called {\em vertices} of $G$), and $\theta_G:{V(G) \choose 2} \rightarrow \{-1,0,1\}$ is a function, called the {\em adjacency function} of $G$. The {\em null} trigraph is the trigraph whose vertex set is empty; a {\em non-null} trigraph is any trigraph whose vertex set is non-empty. If $G$ is a trigraph and $u,v \in V(G)$ are distinct, we usually write $uv$ instead of $\{u,v\}$ (note that this means that $uv = vu$), and furthermore:
\begin{itemize}
\item if $\theta_G(uv) = 1$, we say that $uv$ is a {\em strongly adjacent pair} of $G$, or that $u$ and $v$ are {\em strongly adjacent} in $G$, or that $u$ is {\em strongly adjacent} to $v$ in $G$, or that $v$ is a {\em strong neighbor} of $u$ in $G$, or that $u$ and $v$ are the {\em endpoints of a strongly adjacent pair} of $G$;
\item if $\theta_G(uv) = 0$, we say that $uv$ is a {\em semi-adjacent pair} of $G$, or that $u$ and $v$ are {\em semi-adjacent} in $G$, or that $u$ is {\em semi-adjacent} to $v$ in $G$, or that $v$ is a {\em weak neighbor} of $u$ in $G$, or that $u$ and $v$ are the {\em endpoints of a semi-adjacent pair} of $G$;
\item if $\theta_G(uv) = -1$, we say that $uv$ is a {\em strongly anti-adjacent pair} of $G$, or that $u$ and $v$ are {\em strongly anti-adjacent} in $G$, or that $u$ is {\em strongly anti-adjacent} to $v$ in $G$, or that $v$ is a {\em strong anti-neighbor} of $u$ in $G$, or that $u$ and $v$ are the {\em endpoints of a strongly anti-adjacent pair} of $G$;
\item if $\theta_G(uv) \geq 0$, we say that $uv$ is an {\em adjacent pair} of $G$, or that $u$ and $v$ are {\em adjacent} in $G$, or that $u$ is {\em adjacent} to $v$ in $G$, or that $v$ is a {\em neighbor} of $u$ in $G$, or that $u$ and $v$ are the {\em endpoints of an adjacent pair} of $G$;
\item if $\theta_G(uv) \leq 0$, we say that $uv$ is an {\em anti-adjacent pair} of $G$, or that $u$ and $v$ are {\em anti-adjacent} in $G$, or that $u$ is {\em anti-adjacent} to $v$ in $G$, or that $v$ is an {\em anti-neighbor} of $u$ in $G$, or that $u$ and $v$ are the {\em endpoints of an anti-adjacent pair} of $G$.
\end{itemize}
\noindent
Note that a semi-adjacent pair is simultaneously an adjacent pair and an anti-adjacent pair. One can think of strongly adjacent pairs as ``edges,'' of strongly anti-adjacent pairs as ``non-edges,'' and of semi-adjacent pairs as ``optional edges.'' Clearly, any graph can be thought of as a trigraph: a graph is simply a trigraph with no semi-adjacent pairs, that is, the adjacency function of a graph $G$ is a mapping from ${V(G) \choose 2}$ to the set $\{-1,1\}$.

Given a trigraph $G$, a vertex $u \in V(G)$, and a set $X \subseteq V(G) \smallsetminus \{u\}$, we say that $u$ is {\em complete} (respectively: {\em strongly complete, anti-complete, strongly anti-complete}) to $X$ in $G$ provided that $u$ is adjacent (respectively: strongly adjacent, anti-adjacent, strongly anti-adjacent) to every vertex of $X$ in $G$. Given a trigraph $G$ and disjoint sets $X,Y \subseteq V(G)$, we say that $X$ is {\em complete} (respectively: {\em strongly complete, anti-complete, strongly anti-complete)} to $Y$ in $G$ provided that every vertex of $X$ is complete (respectively: strongly complete, anti-complete, strongly anti-complete) to $Y$ in $G$.

Isomorphism between trigraphs is defined in the natural way. The {\em complement} of a trigraph $G = (V(G),\theta_G)$ is the trigraph $\overline{G} = (V(\overline{G}),\theta_{\overline{G}})$ such that $V(\overline{G}) = V(G)$ and $\theta_{\overline{G}} = -\theta_G$. Thus, $\overline{G}$ is obtained from $G$ by turning all strongly adjacent pairs of $G$ into strongly anti-adjacent pairs, and turning all strongly anti-adjacent pairs of $G$ into strongly adjacent pairs; semi-adjacent pairs of $G$ remain semi-adjacent in $\overline{G}$.

Given trigraphs $G$ and $\widetilde{G}$, we say that $\widetilde{G}$ is a {\em semi-realization} of $G$ provided that $V(\widetilde{G}) = V(G)$
and for all distinct $u,v \in V(\widetilde{G}) = V(G)$, we have that if $\theta_G(uv) = 1$ then $\theta_{\widetilde{G}}(uv) = 1$, and if $\theta_G(uv) = -1$ then $\theta_{\widetilde{G}}(uv) = -1$. Thus, a semi-realization of a trigraph $G$ is any trigraph that can be obtained from $G$ by ``deciding'' the adjacency of some semi-adjacent pairs of $G$, that is, by possibly turning some semi-adjacent pairs of $G$ into strongly adjacent or strongly anti-adjacent pairs. (In particular, every trigraph is a semi-realization of itself.) A {\em realization} of a trigraph $G$ is a graph that is a semi-realization of $G$. Thus, a realization of a trigraph $G$ is any graph that can be obtained by ``deciding'' the adjacency of all semi-adjacent pairs of $G$, that is, by turning each semi-adjacent pair of $G$ into an edge or a non-edge. Clearly, if a trigraph $G$ has $m$ semi-adjacent pairs, then $G$ has $3^m$ semi-realizations and $2^m$ realizations. The {\em full realization} of a trigraph $G$ is the graph obtained from $G$ by turning all semi-adjacent pairs of $G$ into strongly adjacent pairs (i.e., edges), and the {\em null realization} of $G$ is the graph obtained from $G$ by turning all semi-adjacent pairs of $G$ into strongly anti-adjacent pairs (i.e., non-edges).

A {\em clique} (respectively: {\em strong clique}, {\em stable set}, {\em strongly stable set}) of a trigraph $G$ is a set of pairwise adjacent (respectively: strongly adjacent, anti-adjacent, strongly anti-adjacent) vertices of $G$. Note that any subset of $V(G)$ of size at most one is both a strong clique and a strongly stable set of $G$. Note also that if $S \subseteq V(G)$, then $S$ is a (strong) clique of $G$ if and only if $S$ is a (strongly) stable set of $\overline{G}$. Note furthermore that if $K$ is a strong clique and $S$ is a stable set of $G$, then $|K \cap S| \leq 1$; similarly, if $K$ is a clique and $S$ is a strongly stable set of $G$, then $|K \cap S| \leq 1$. However, if $K$ is a clique and $S$ is a stable set of $G$, then we are only guaranteed that vertices in $K \cap S$ are pairwise semi-adjacent to each other, and it is possible that $|K \cap S| \geq 2$. A {\em triangle} (respectively: {\em strong triangle}) is a clique (respectively: strong clique) of size three.

Given a trigraph $G$ and a set $X \subseteq V(G)$, the {\em subtrigraph of $G$ induced by $X$}, denoted by $G[X]$, is the trigraph with vertex set $X$ and adjacency function $\theta_G \upharpoonright {X \choose 2}$, where for a function $f:A\to B$ and a set $A'\subseteq A$, we denote by $f\upharpoonright  A'$ the restriction of $f$ to $A'$. If $H = G[X]$ for some $X \subseteq V(G)$, we also say that $H$ is an {\em induced subtrigraph} of $G$; when convenient, we relax this definition and say that $H$ is an induced subtrigraph of $G$ provided that there is some set $X \subseteq V(G)$ such that $H$ is isomorphic to $G[X]$. If $v_1,\ldots,v_k$ are vertices of a trigraph $G$, we often write $G[v_1,\ldots,v_k]$ instead of $G[\{v_1,\ldots,v_k\}]$. Further, for a trigraph $G$ and a set $X \subseteq V(G)$, we set $G \smallsetminus X = G[V(G) \smallsetminus X]$; for $v \in V(G)$, we often write $G \smallsetminus v$ instead of $G \smallsetminus \{v\}$. The trigraph $G \smallsetminus X$ (respectively: $G \smallsetminus v$) is called the subtrigraph of $G$ obtained by {\em deleting} $X$ (respectively: by {\em deleting} $v$).

If $H$ is a graph, we say that a trigraph $G$ is an {\em $H$-trigraph} if some realization of $G$ is (isomorphic to) $H$. Further, if $H$ is a graph and $G$ a trigraph, we say that $G$ is {\em $H$-free} provided that all realizations of $G$ are $H$-free (equivalently: provided that no induced subtrigraph of $G$ is an $H$-trigraph). If $\mathcal{H}$ is a family of graphs, we say that a trigraph $G$ is {\em $\mathcal{H}$-free} provided that $G$ is $H$-free for all graphs $H \in \mathcal{H}$. In particular, a trigraph is {\em ISK4-free} (respectively: {\em wheel-free}, {\em \{ISK4,wheel\}-free}) if all its realizations are ISK4-free (respectively: wheel-free, \{ISK4,wheel\}-free).

A trigraph is {\em connected} if its full realization is a connected graph. A trigraph is {\em disconnected} if it is not connected. A {\em component} of a non-null trigraph $G$ is any (inclusion-wise) vertex-maximal connected induced subtrigraph of $G$. Clearly, if $H$ is an induced subtrigraph of a non-null trigraph $G$, then we have that $H$ is a component of $G$ if and only if the full realization of $H$ is a component of the full realization of $G$.

A trigraph is a {\em path} if at least one of its realizations is a path. A trigraph is a {\em narrow path} if its full realization is a path. We often denote a path $P$ by $v_0-v_1-\dots-v_k$ (with $k \geq 0$), where $v_0,v_1,\dots,v_k$ are the vertices of $P$ that appear in that order in some realization $\widetilde{P}$ of $P$ such that $\widetilde{P}$ is a path. The {\em endpoints} of a narrow path are the endpoints of its full realization; if $a$ and $b$ are the endpoints of a narrow path $P$, then we also say that $P$ is a narrow path {\em between} $a$ and $b$. A {\em path} (respectively: {\em narrow path}) in a trigraph $G$ is an induced subtrigraph $P$ of $G$ such that $P$ is a path (respectively: narrow path).

Note that if $G$ is a connected trigraph, then for all vertices $a,b \in V(G)$, there exists a narrow path between $a$ and $b$ in $G$. (To see this, consider the full realization $\widetilde{G}$ of $G$. $\widetilde{G}$ is connected, and so there is a path in $\widetilde{G}$ between $a$ and $b$; let $P$ be a shortest such path in $\widetilde{G}$. The minimality of $P$ guarantees that $P$ is an induced path of $\widetilde{G}$. But now $G[V(P)]$ is a narrow path of $G$ between $a$ and $b$.)

A {\em hole} of a trigraph $G$ is an induced subtrigraph $C$ of $G$ such that some realization of $C$ is a chordless cycle of length at least four. We often denote a hole $C$ of $G$ by $v_0-v_1-\dots-v_{k-1}-v_0$ (with $k \geq 4$ and indices in $\mathbb{Z}_k$), where $v_0,v_1,\dots,v_{k-1}$ are the vertices of $C$ that appear in that order in some realization $\widetilde{C}$ of $C$ such that $\widetilde{C}$ is a chordless cycle of length at least four.

A {\em cutset} of a trigraph $G$ is a (possibly empty) set $C \subseteq V(G)$ such that $G \smallsetminus C$ is disconnected. A {\em cut-partition} of a trigraph $G$ is a partition $(A,B,C)$ of $V(G)$ such that $A$ and $B$ are non-empty ($C$ may possibly be empty), and $A$ is strongly anti-complete to $B$. Note that if $(A,B,C)$ is a cut-partition of $G$, then $C$ is a cutset of $G$. Conversely, every cutset of $G$ induces at least one cut-partition of $G$. A {\em clique-cutset} of a trigraph $G$ is a (possibly empty) strong clique $C$ of $G$ such that $G \smallsetminus C$ is disconnected. A {\em cut-vertex} of a trigraph $G$ is a vertex $v \in V(G)$ such that $G \smallsetminus v$ is disconnected. Note that if $v$ is a cut-vertex of $G$, then $\{v\}$ is a clique-cutset of $G$. A {\em stable 2-cutset} of a trigraph $G$ is cutset of $G$ that is a stable set of size two. We remark that if $C$ is a cutset of a trigraph $G$ such that $|C| \leq 2$, then $C$ is either a clique-cutset or a stable 2-cutset of $G$.

A graph is {\em series-parallel} if it does not contain any subdivision of $K_4$ as a (not necessarily induced) subgraph. A trigraph is {\em series-parallel} if its full realization is series-parallel (equivalently: if all its realizations are series-parallel).

A {\em bipartite trigraph} is a trigraph $G$ whose vertex set can be partitioned into two (possibly empty) strongly stable sets, $A$ and $B$; under these circumstances, $(A,B)$ is said to be a {\em bipartition} of the bipartite trigraph $G$. 
If, in addition, the two strongly stable sets $A$ and $B$ forming a bipartition are strongly complete to each other, $G$ is said to be
a {\em complete bipartite trigraph}.
Note that non-null complete bipartite trigraphs have precisely two bipartitions: if $(A,B)$ is a bipartition of a complete bipartite trigraph $G$, then so is $(B,A)$, and $G$ has no other bipartitions. Furthermore, note that bipartite trigraphs may have semi-adjacent pairs, but complete bipartite trigraphs cannot. Thus, complete bipartite trigraphs are in fact complete bipartite graphs.

The {\em line graph} of a graph $H$, denoted by $L(H)$, is the graph whose vertices are the edges of $H$, and in which two vertices (i.e., edges of $H$) are adjacent if they share an endpoint in $H$. A {\em line trigraph} of a graph $H$ is a trigraph $G$ whose full realization is (isomorphic to) $L(H)$, and all of whose triangles are strong. A trigraph $G$ is said to be a {\em line trigraph} provided there is a graph $H$ such that $G$ is a line trigraph of $H$.

\section{Stable sets in weighted trigraphs}
\label{s:swi}

In what follows, $\mathbb{N}$ is the set of non-negative integers. Given a trigraph $G$, we define
\begin{displaymath}
\begin{array}{rcl}
D(G) & = & V(G) \cup \{(u,v) \mid u,v \in V(G), u \neq v\} \cup {V(G) \choose 2}.
\end{array}
\end{displaymath}
A {\em weight function} for a trigraph $G$ is any function $w:D(G) \rightarrow \mathbb{N}$ that satisfies the following two properties:
\begin{itemize}
\item for all distinct $u,v \in V(G)$, if $uv$ is not a semi-adjacent pair of $G$, then $w(u,v) = w(v,u) = w(uv) = 0$;
\item all distinct $u,v \in V(G)$ satisfy $w(u,v)\leq w(uv)$.
\end{itemize}
A {\em weighted trigraph} is an ordered pair $(G,w)$ where $G$ is a trigraph and $w$ is a weight function for $G$.

\begin{sloppypar}
Essentially, a weight function $w$ assigns a non-negative integer weight $w(u)$ to each vertex $u$ of the trigraph $G$, and for each semi-adjacent pair $uv$, there are three non-negative integer weights associated with it, namely $w(u,v)$, $w(v,u)$, and $w(uv)$, and these weights must satisfy $\max\{w(u,v),w(v,u)\} \leq w(uv)$. If $uv$ is a strongly adjacent or strongly anti-adjacent pair, then we have $w(u,v) = w(v,u) = w(uv) = 0$. (Zero weights are assigned to strongly adjacent and strongly anti-adjacent pairs for the purposes of making calculations notationally simpler, but only vertices and semi-adjacent pairs actually ``count.'')
\end{sloppypar}

Note that if a trigraph $G$ is a semi-realization of a trigraph $G'$, then every weight function for $G$ is also a weight function for $G'$ (however, a weight function for $G'$ need not be a weight function for $G$).

If $(G,w)$ is a weighted trigraph, and $H$ is an induced subtrigraph of $G$, then clearly, $(H,w \upharpoonright D(H))$ is also a weighted trigraph; to simplify notation, we often write $(H,w)$ instead of $(H,w \upharpoonright D(H))$.

Given a weighted trigraph $(G,w)$ and a set $S \subseteq V(G)$, the {\em weight} of $S$ with respect to $(G,w)$, denoted by $\llbracket S \rrbracket_{(G,w)}$, is defined to be \begin{displaymath}
\begin{array}{rcl}
\llbracket S \rrbracket_{(G,w)} & = & \Big(\sum\limits_{u \in S} w(u)\Big)+\Big(\sum\limits_{u \in S} \sum\limits_{v \in V(G) \smallsetminus S} w(u,v)\Big)+\Big(\sum\limits_{uv \in {V(G) \smallsetminus S \choose 2}} w(uv)\Big).
\end{array}
\end{displaymath}
Note that if $(G,w)$ is a weighted trigraph such that $G$ has no semi-adjacent pairs (that is, such that $G$ is a graph), then for all $S \subseteq V(G)$, we have that $\llbracket S \rrbracket_{(G,w)} = \sum\limits_{u \in S} w(u)$. Thus, our definition of a weight of a set of vertices in a weighted trigraph indeed generalizes the usual notion of the weight of a set in a weighted graph.

It is easy to see that for all weighted trigraphs $(G,w)$, all induced subtrigraphs $H$ of $G$, and all sets $S \subseteq V(H)$, we have that $\llbracket S \rrbracket_{(H,w)} \leq \llbracket S \rrbracket_{(G,w)}$. Strict inequality may hold because the weight of a set in a weighted trigraph depends not only on what is in the set, but also on what is outside of it. Furthermore, if $(G,w)$ is a weighted trigraph and $S_1 \subsetneqq S_2 \subseteq V(G)$, there is in general no relationship between $\llbracket S_1 \rrbracket_{(G,w)}$ and $\llbracket S_2 \rrbracket_{(G,w)}$, that is, any one of the following is possible: $\llbracket S_1 \rrbracket_{(G,w)} < \llbracket S_2 \rrbracket_{(G,w)}$, $\llbracket S_1 \rrbracket_{(G,w)} = \llbracket S_2 \rrbracket_{(G,w)}$, and $\llbracket S_1 \rrbracket_{(G,w)} > \llbracket S_2 \rrbracket_{(G,w)}$.

The {\em stability number} of a weighted trigraph $(G,w)$, denoted by $\alpha(G,w)$, is defined to be
\begin{displaymath}
\begin{array}{rcl}
\alpha(G,w) & = & \max\{\llbracket S \rrbracket_{(G,w)} \mid \text{$S$ is a stable set of $G$}\}.
\end{array}
\end{displaymath}

A {\em zero-vertex} of a weighted trigraph $(G,w)$ is any vertex $u \in V(G)$ such that $w(u) = 0$.

\begin{proposition} \label{prop-delete-vertices-weight} Let $(G,w)$ be a weighted trigraph, and let $Z,S \subseteq V(G)$. Then $\llbracket S \rrbracket_{(G,w)} \leq \llbracket S \smallsetminus Z \rrbracket_{(G,w)}+\sum\limits_{u \in Z} w(u)$.
\end{proposition}
\begin{proof}
Since $w(u) \geq 0$ for all $u \in V(G)$, we may assume that $Z \subseteq S$. Using the definition of $\llbracket S \rrbracket_{(G,w)}$ and $\llbracket S \smallsetminus Z \rrbracket_{(G,w)}$, we obtain the following:
\begin{displaymath}
\begin{array}{rcl}
\llbracket S \rrbracket_{(G,w)} & = & \Big(\sum\limits_{u \in S} w(u)\Big)+\Big(\sum\limits_{u \in S} \sum\limits_{v \in V(G) \smallsetminus S} w(u,v)\Big)+
\\
& & +\Big(\sum\limits_{uv \in {V(G) \smallsetminus S \choose 2}} w(uv)\Big)
\\
\\
& = & \Big(\sum\limits_{u \in S \smallsetminus Z} w(u)\Big)+\Big(\sum\limits_{u \in Z} w(u)\Big)+
\\
& & +\Big(\sum\limits_{u \in S \smallsetminus Z} \sum\limits_{v \in V(G) \smallsetminus (S \smallsetminus Z)} w(u,v)\Big)-\Big(\sum\limits_{u \in S \smallsetminus Z} \sum\limits_{v \in Z} w(u,v)\Big)+
\\
& & +\Big(\sum\limits_{u \in Z} \sum\limits_{v \in V(G) \smallsetminus S} w(u,v)\Big)+\Big(\sum\limits_{uv \in {V(G) \smallsetminus (S \smallsetminus Z) \choose 2}} w(uv)\Big)-
\\
& & -\Big(\sum\limits_{uv \in {Z \choose 2}} w(uv)\Big)-\Big(\sum\limits_{u \in Z} \sum\limits_{v \in V(G) \smallsetminus S} w(uv)\Big)
\\
\\
& = & \llbracket S \smallsetminus Z \rrbracket_{(G,w)}+\Big(\sum\limits_{u \in Z} w(u)\Big)-\Bigg(\Big(\sum\limits_{u \in S \smallsetminus Z} \sum\limits_{v \in Z} w(u,v)\Big)+
\\
& & +\Big(\sum\limits_{u \in Z} \sum\limits_{v \in V(G) \smallsetminus S} (w(uv)-w(u,v))\Big)+\Big(\sum\limits_{uv \in {Z \choose 2}} w(uv)\Big)\Bigg).
\end{array}
\end{displaymath}
By the definition of a weight function, we have that $w(uv) \geq w(u,v) \geq 0$ for all distinct $u,v \in V(G)$. The calculation above now implies that $\llbracket S \rrbracket_{(G,w)} \leq \llbracket S \smallsetminus Z \rrbracket_{(G,w)}+\sum\limits_{u \in Z} w(u)$, which is what we needed.
\end{proof}
\begin{proposition} \label{prop-delete-zero} For all weighted trigraphs $(G,w)$, there exists a stable set $S$ of $G$ such that $S$ contains no zero-vertices of $(G,w)$ and $\llbracket S \rrbracket_{(G,w)} = \alpha(G,w)$.
\end{proposition}
\begin{proof}
Fix a weighted trigraph $(G,w)$ and a stable set $S$ of $G$ such that $\llbracket S \rrbracket_{(G,w)} = \alpha(G,w)$. Let $Z$ be the set of all zero-vertices of $G$. Then $S \smallsetminus Z$ is a stable set of $G$ that contains no zero vertices of $G$, and clearly, we have that $\llbracket S \smallsetminus Z \rrbracket_{(G,w)} \leq \alpha(G,w)$. On the other hand, Proposition~\ref{prop-delete-vertices-weight} implies that $\alpha(G,w) = \llbracket S \rrbracket_{(G,w)} \leq \llbracket S \smallsetminus Z \rrbracket_{(G,w)}+\sum\limits_{u \in Z} w(u) = \llbracket S \smallsetminus Z \rrbracket_{(G,w)}$. It follows that $\llbracket S \smallsetminus Z \rrbracket_{(G,w)} = \alpha(G,w)$, and so $S \smallsetminus Z$ is the stable set that we needed.
\end{proof}

The next two propositions (Propositions~\ref{weight-cut-part} and~\ref{differ-at-C-only}) are easy consequences of the appropriate definitions, and we leave their proofs as exercises for the reader.
\begin{proposition} \label{weight-cut-part} Let $(G,w)$ be a weighted trigraph, let $(A,B,C)$ be a cut-partition of $G$, and let $S \subseteq V(G)$. Then $\llbracket S \cap (A \cup C) \rrbracket_{(G[A \cup C],w)}+\llbracket S \cap (B \cup C) \rrbracket_{(G[B \cup C],w)} = \llbracket S \rrbracket_{(G,w)}+\llbracket S \cap C \rrbracket_{(G[C],w)}$.
\end{proposition}

\begin{proposition} \label{differ-at-C-only} Let $(G,w)$ and $(G',w')$ be weighted trigraphs such that $V(G) = V(G')$. Let $C \subseteq V(G)$, and assume that $\theta_G \upharpoonright ({V(G) \choose 2} \smallsetminus {C \choose 2}) = \theta_{G'} \upharpoonright ({V(G) \choose 2} \smallsetminus {C \choose 2})$ and $w \upharpoonright (D(G) \smallsetminus D(G[C])) = w' \upharpoonright (D(G') \smallsetminus D(G'[C]))$ (that is, adjacency and weights in $(G,w)$ and $(G',w')$ are the same except possibly within $C$). Let $S \subseteq V(G)$. Then $\llbracket S \rrbracket_{(G,w)}-\llbracket S \cap C \rrbracket_{(G[C],w)} = \llbracket S \rrbracket_{(G',w')}-\llbracket S \cap C \rrbracket_{(G'[C],w')}$.
\end{proposition}

\begin{sloppypar}
We now need a couple of definitions. If $(G,w)$ is a weighted trigraph and $R \subseteq V(G)$, the {\em reduction} of $(G,w)$ to $R$, denoted by ${\rm Red}[G,w;R]$, is defined to be the weighted trigraph $(G[R],w')$, where $w':D(G[R]) \rightarrow \mathbb{N}$ is given by:
\begin{itemize}
\item for all $u \in R$, $w'(u) = \max\left\{w(u)-\sum\limits_{v \in V(G) \smallsetminus R} (w(uv)-w(u,v)),0\right\}$;
\item for all distinct $u,v \in R$, $w'(u,v) = w(u,v)$;
\item for all $uv \in {R \choose 2}$, $w'(uv) = w(uv)$.
\end{itemize}
Further, we define the {\em exterior weight} of $R$ with respect to $(G,w)$, denoted by ${\rm Ext}[G,w;R]$, to be
\begin{displaymath}
\begin{array}{rcl}
{\rm Ext}[G,w;R] & = & \Big(\sum\limits_{uv \in {V(G) \smallsetminus R \choose 2}} w(uv)\Big)+\Big(\sum\limits_{u \in R}\sum\limits_{v \in V(G) \smallsetminus R} w(uv)\Big).
\end{array}
\end{displaymath}
We remark that for all weighted trigraphs $(G,w)$, we have that ${\rm Red}[G,w;V(G)] = (G,w)$ and ${\rm Ext}[G,w;V(G)] = 0$, and consequently, $\alpha(G,w) = \alpha({\rm Red}[G,w;V(G)])+{\rm Ext}[G,w;V(G)]$.
\end{sloppypar}

\vbox{\begin{proposition} \label{prop-ext-alg} There is an algorithm with the following specifications:
\begin{itemize}
\item Input: a weighted trigraph $(G,w)$ and a set $R \subseteq V(G)$;
\item Output: ${\rm Red}[G,w;R]$ and ${\rm Ext}[G,w;R]$;
\item Running time: $O(n^2)$, where $n = |V(G)|$.
\end{itemize}
\end{proposition}}

\begin{proof}
Clearly, the trigraph $G[R]$ can be computed in time $O(n^2)$. Similarly, the quantity $\sum\limits_{uv \in {V(G) \smallsetminus R \choose 2}} w(uv)$ can be found in time $O(n^2)$. Further, for each vertex $u \in R$, the quantities $\sum\limits_{v \in V(G) \smallsetminus R} w(uv)$ and $\max\{w(u)-\sum\limits_{v \in V(G) \smallsetminus R} (w(uv)-w(u,v)),0\}$ can be found in $O(n)$ time. Since $R$ contains at most $n$ vertices, the result follows.
\end{proof}

\begin{proposition} \label{prop-ext} Let $(G,w)$ be a weighted trigraph, and let $S \subseteq R \subseteq V(G)$. Then $\llbracket S \rrbracket_{(G,w)} \leq \llbracket S \rrbracket_{{\rm Red}[G,w;R]}+{\rm Ext}[G,w;R]$. Furthermore, if $S$ contains no zero-vertices of ${\rm Red}[G,w;R]$, then equality holds, that is, $\llbracket S \rrbracket_{(G,w)} = \llbracket S \rrbracket_{{\rm Red}[G,w;R]}+{\rm Ext}[G,w;R]$.
\end{proposition}
\begin{proof}
Set $w':D(G[R]) \rightarrow \mathbb{N}$ so that $(G[R],w') = {\rm Red}[G,w;R]$. By definition, for all $u \in S$, $w'(u) \geq w(u)-\sum\limits_{v \in V(G) \smallsetminus R} (w(uv)-w(u,v))$ (and if $u$ is not a zero-vertex of ${\rm Red}[G,w;R]$, then equality holds). Consequently,
\begin{displaymath}
\begin{array}{ll}
& \llbracket S \rrbracket_{{\rm Red}[G,w;R]}+{\rm Ext}[G,w;R]
\\
\\
= & \Big(\sum\limits_{u \in S} w'(u)\Big)+\Big(\sum\limits_{u \in S} \sum\limits_{v \in R \smallsetminus S} w'(u,v)\Big)+\Big(\sum\limits_{uv \in {R \smallsetminus S \choose 2}} w'(uv)\Big)+
\\
& +\Big(\sum\limits_{uv \in {V(G) \smallsetminus R \choose 2}} w(uv)\Big)+\Big(\sum\limits_{u \in R}\sum\limits_{v \in V(G) \smallsetminus R} w(uv)\Big)
\\
\\
\geq & \Bigg(\sum\limits_{u \in S} \Big(w(u)-\sum\limits_{v \in V(G) \smallsetminus R} (w(uv)-w(u,v))\Big)\Bigg)+\Big(\sum\limits_{u \in S} \sum\limits_{v \in R \smallsetminus S} w(u,v)\Big)+
\\
& +\Big(\sum\limits_{uv \in {R \smallsetminus S \choose 2}} w(uv)\Big)+\Big(\sum\limits_{uv \in {V(G) \smallsetminus R \choose 2}} w(uv)\Big)+\Big(\sum\limits_{u \in R}\sum\limits_{v \in V(G) \smallsetminus R} w(uv)\Big)
\\
\\
= & \Big(\sum\limits_{u \in S} w(u)\Big)-\Big(\sum\limits_{u \in S}\sum\limits_{v \in V(G) \smallsetminus R} w(uv)\Big)+\Big(\sum\limits_{u \in S} \sum\limits_{v \in V(G) \smallsetminus S} w(u,v)\Big)+
\\
& +\Big(\sum\limits_{uv \in {R \smallsetminus S \choose 2}} w(uv)\Big)+\Big(\sum\limits_{uv \in {V(G) \smallsetminus R \choose 2}} w(uv)\Big)+\Big(\sum\limits_{u \in R}\sum\limits_{v \in V(G) \smallsetminus R} w(uv)\Big)
\\
\\
= & \Big(\sum\limits_{u \in S} w(u)\Big)+\Big(\sum\limits_{u \in S} \sum\limits_{v \in V(G) \smallsetminus S} w(u,v)\Big)+\Big(\sum\limits_{uv \in {R \smallsetminus S \choose 2}} w(uv)\Big)+
\\
& +\Big(\sum\limits_{uv \in {V(G) \smallsetminus R \choose 2}} w(uv)\Big)+\Big(\sum\limits_{u \in R \smallsetminus S}\sum\limits_{v \in V(G) \smallsetminus R} w(uv)\Big)
\\
\\
= & \Big(\sum\limits_{u \in S} w(u)\Big)+\Big(\sum\limits_{u \in S} \sum\limits_{v \in V(G) \smallsetminus S} w(u,v)\Big)+\Big(\sum\limits_{uv \in {V(G) \smallsetminus S \choose 2}} w(uv)\Big)
\\
\\
= & \llbracket S \rrbracket_{(G,w)}.
\end{array}
\end{displaymath}
This proves that $\llbracket S \rrbracket_{(G,w)} \leq \llbracket S \rrbracket_{{\rm Red}[G,w;R]}+{\rm Ext}[G,w;R]$. Furthermore, if $S$ contains no zero vertices of ${\rm Red}[G,w;R]$ (and so $w'(u) = w(u)-\sum\limits_{v \in V(G) \smallsetminus R} (w(uv)-w(u,v))$ for all $u \in S$), the computation above yields $\llbracket S \rrbracket_{(G,w)} = \llbracket S \rrbracket_{{\rm Red}[G,w;R]}+{\rm Ext}[G,w;R]$.
\end{proof}

\begin{proposition} \label{prop-S-red-ext} Let $(G,w)$ be a weighted trigraph, let $S \subseteq R \subseteq V(G)$, and assume that $S$ is a stable set of $G$. Then $\sum\limits_{u \in S} w(u) \leq \alpha({\rm Red}[G,w;R])+{\rm Ext}[G,w;R]$.
\end{proposition}
\begin{proof}
By the definition of $\llbracket S \rrbracket_{(G,w)}$, we have that $\sum\limits_{u \in S} w(u) \leq \llbracket S \rrbracket_{(G,w)}$. We now compute:
\begin{displaymath}
\begin{array}{rcllll}
\sum\limits_{u \in S} w(u) & \leq & \llbracket S \rrbracket_{(G,w)}
\\
\\
& \leq & \llbracket S \rrbracket_{{\rm Red}[G,w;R]}+{\rm Ext}[G,w;R] & & & \text{by Proposition~\ref{prop-ext}}
\\
\\
& \leq & \alpha({\rm Red}[G,w;R])+{\rm Ext}[G,w;R].
\end{array}
\end{displaymath}
Thus, $\sum\limits_{u \in S} w(u) \leq \alpha({\rm Red}[G,w;R])+{\rm Ext}[G,w;R]$. This completes the argument.
\end{proof}

\begin{proposition} \label{prop-alpha-nested} Let $(G,w)$ be a weighted trigraph, and let $R_1,R_2 \subseteq V(G)$ be disjoint sets. Set $\alpha_{R_1} = \alpha({\rm Red}[G,w;R_1])+{\rm Ext}[G,w;R_1]$ and $\alpha_{R_1 \cup R_2} = \alpha({\rm Red}[G,w;R_1 \cup R_2])+{\rm Ext}[G,w;R_1 \cup R_2]$. Then $\alpha_{R_1} \leq \alpha_{R_1 \cup R_2} \leq \alpha_{R_1}+\sum\limits_{u \in R_2} w(u)$.
\end{proposition}
\begin{proof}
We first show that $\alpha_{R_1} \leq \alpha_{R_1 \cup R_2}$. Using Proposition~\ref{prop-delete-zero}, we fix a stable set $S \subseteq R_1$ of $G$ that contains no zero-vertices of ${\rm Red}[G,w;R_1]$ and satisfies $\llbracket S \rrbracket_{{\rm Red}[G,w;R_1]} = \alpha({\rm Red}[G,w;R_1])$. Then
\begin{displaymath}
\begin{array}{rcllll}
\alpha_{R_1} & = & \alpha({\rm Red}[G,w;R_1])+{\rm Ext}[G,w;R_1] & & &
\\
\\
& = & \llbracket S \rrbracket_{{\rm Red}[G,w;R_1]}+{\rm Ext}[G,w;R_1] & & &
\\
\\
& = & \llbracket S \rrbracket_{(G,w)} & & & \text{by Proposition~\ref{prop-ext}}
\\
\\
& \leq & \llbracket S \rrbracket_{{\rm Red}[G,w;R_1 \cup R_2]}+{\rm Ext}[G,w;R_1 \cup R_2] & & & \text{by Proposition~\ref{prop-ext}}
\\
\\
& \leq & \alpha({\rm Red}[G,w;R_1 \cup R_2])+
\\
& & +{\rm Ext}[G,w;R_1 \cup R_2]
\\
\\
& = & \alpha_{R_1 \cup R_2}.
\end{array}
\end{displaymath}
Thus, $\alpha_{R_1} \leq \alpha_{R_1 \cup R_2}$.

It remains to show that $\alpha_{R_1 \cup R_2} \leq \alpha_{R_1}+\sum\limits_{u \in R_2} w(u)$. Using Proposition~\ref{prop-delete-zero}, we fix a stable set $S \subseteq R_1 \cup R_2$ that contains no zero-vertices of ${\rm Red}[G,w;R_1 \cup R_2]$ and satisfies $\llbracket S \rrbracket_{{\rm Red}[G,w;R_1 \cup R_2]} = \alpha({\rm Red}[G,w;R_1 \cup R_2])$. We then have the following:
\begin{displaymath}
\begin{array}{rcllll}
\alpha_{R_1 \cup R_2} & = & \alpha({\rm Red}[G,w;R_1 \cup R_2])+
\\
& & +{\rm Ext}[G,w;R_1 \cup R_2]
\\
\\
& = & \llbracket S \rrbracket_{{\rm Red}[G,w;R_1 \cup R_2]}+
\\
& & +{\rm Ext}[G,w;R_1 \cup R_2]
\\
\\
& = & \llbracket S \rrbracket_{(G,w)} & & & \text{by Proposition~\ref{prop-ext}}
\\
\\
& \leq & \llbracket S \smallsetminus R_2 \rrbracket_{(G,w)}+\sum\limits_{u \in R_2} w(u) & & & \text{by Proposition~\ref{prop-delete-vertices-weight}}
\\
\\
& \leq & \llbracket S \smallsetminus R_2 \rrbracket_{{\rm Red}[G,w;R_1]}+ & & & \text{by Proposition~\ref{prop-ext}}
\\
& & +{\rm Ext}[G,w;R_1]+\sum\limits_{u \in R_2} w(u)
\\
\\
& \leq & \alpha({\rm Red}[G,w;R_1])+
\\
& & +{\rm Ext}[G,w;R_1]+\sum\limits_{u \in R_2} w(u)
\\
\\
& = & \alpha_{R_1}+\sum\limits_{u \in R_2} w(u).
\end{array}
\end{displaymath}
Thus, $\alpha_{R_1 \cup R_2} \leq \alpha_{R_1}+\sum\limits_{u \in R_2} w(u)$. This completes the argument.
\end{proof}

Before stating our next proposition, we remind the reader that if $G$ is a semi-realization of a trigraph $G'$, then every weight function for $G$ is also a weight function for $G'$. In particular then, if $(G,w)$ is a weighted trigraph, and $G'$ is a trigraph obtained from $G$ by possibly turning some strongly anti-adjacent pairs of $G$ into semi-adjacent pairs, then $(G',w)$ is also a weighted trigraph.

\begin{proposition} \label{prop-cut-part-reduction} Let $(G,w)$ be a weighted trigraph, and let $(A,B,C)$ be a cut-partition of $G$. For each $X \in \{A,B\}$, let $G_X$ be a trigraph obtained from $G[X \cup C]$ by possibly turning some strongly anti-adjacent pairs of $G[X \cup C]$ into semi-adjacent pairs. For all $C' \subseteq C$, set $\alpha_{A \cup C'} = \alpha({\rm Red}[G_A,w;A \cup C'])+{\rm Ext}[G_A,w;A \cup C']$. Let $k \in \mathbb{N}$ and let $w_B$ be a weight function for $G_B$ that satisfies all the following:
\begin{itemize}
\item for all $u \in B$, $w_B(u) = w(u)$;
\item for all $uv \in {B \cup C \choose 2} \smallsetminus {C \choose 2}$, $w_B(u,v) = w(u,v)$ and $w_B(uv) = w(uv)$;
\item for all $S_C \subseteq C$ such that $S_C$ is a stable set of $G_B$, we have that $\llbracket S_C \rrbracket_{(G_B[C],w_B)} = \alpha_{A \cup S_C}-k$.
\end{itemize}
Then $\alpha(G,w) = k+\alpha(G_B,w_B)$.
\end{proposition}
\begin{proof}
We begin by observing that for all $X \in \{A,B\}$ and $S \subseteq X \cup C$, we have that $S$ is a stable set of $G_X$ if and only if $S$ is a stable set of $G[X \cup C]$, and furthermore, for all $Y \subseteq X \cup C$, we have that $\llbracket S \cap Y \rrbracket_{(G_X[Y],w)} = \llbracket S \cap Y \rrbracket_{(G[Y],w)}$.

Let us first show that $\alpha(G,w) \leq k+\alpha(G_B,w_B)$. Fix a stable set $S$ of $G$ such that $\llbracket S \rrbracket_{(G,w)} = \alpha(G,w)$. Set $S_A = S \cap (A \cup C)$, $S_B = S \cap (B \cup C)$, and $S_C = S \cap C$. We then have the following:
\begin{displaymath}
\begin{array}{rcllll}
\alpha(G,w) & = & \llbracket S \rrbracket_{(G,w)}
\\
\\
& = & \llbracket S_A \rrbracket_{(G[A \cup C],w)}+\llbracket S_B \rrbracket_{(G[B \cup C],w)}- & & & \text{by Proposition~\ref{weight-cut-part}}
\\
& & -\llbracket S_C \rrbracket_{(G[C],w)}
\\
\\
& = & \llbracket S_A \rrbracket_{(G_A,w)}+\llbracket S_B \rrbracket_{(G_B,w)}- & & &
\\
& & -\llbracket S_C \rrbracket_{(G_B[C],w)}
\\
\\
& = & \llbracket S_A \rrbracket_{(G_A,w)}+\llbracket S_B \rrbracket_{(G_B,w_B)}- & & & \text{by Proposition~\ref{differ-at-C-only}}
\\
& & -\llbracket S_C \rrbracket_{(G_B[C],w_B)}
\\
\\
& = & \llbracket S_A \rrbracket_{(G_A,w)}+\llbracket S_B \rrbracket_{(G_B,w_B)}- & & &
\\
& & -(\alpha_{A \cup S_C}-k)
\\
\\
& \leq & k+\alpha(G_B,w_B)-\alpha_{A \cup S_C}+
\\
& & +\llbracket S_A \rrbracket_{(G_A,w)}
\\
\\
& \leq & k+\alpha(G_B,w_B)-\alpha_{A \cup S_C}+ & & & \text{by Proposition~\ref{prop-ext}}
\\
& & +\llbracket S_A \rrbracket_{{\rm Red}[G_A,w;A \cup S_C]}+
\\
& & +{\rm Ext}[G_A,w;A \cup S_C]
\\
\\
& \leq & k+\alpha(G_B,w_B)-\alpha_{A \cup S_C}+
\\
& & +\alpha({\rm Red}[G_A,w;A \cup S_C])+
\\
& & +{\rm Ext}[G_A,w;A \cup S_C]
\\
\\
& = & k+\alpha(G_B,w_B).
\end{array}
\end{displaymath}
This proves that $\alpha(G,w) \leq k+\alpha(G_B,w_B)$.

It remains to show that $k+\alpha(G_B,w_B) \leq \alpha(G,w)$. Using Proposition~\ref{prop-delete-zero}, we fix a stable set $S_B$ of $G_B$ that contains no zero-vertices of $G_B$ and satisfies $\llbracket S_B \rrbracket_{(G_B,w_B)} = \alpha(G_B,w_B)$; we may assume that $S_B$ was chosen inclusion-minimal with this property, that is, that for all $S_B' \subsetneqq S_B$, we have that $\llbracket S_B' \rrbracket_{(G_B,w_B)} < \alpha(G_B,w_B)$. Set $S_C = S_B \cap C$.

Let us first check that for all $S_C' \subsetneqq S_C$, we have that $\alpha_{A \cup S_C'} < \alpha_{A \cup S_C}$. Fix $S_C' \subsetneqq S_C$, and set $S_B' = (S_B \smallsetminus C) \cup S_C'$. By the minimality of $S_B$, we have that $\llbracket S_B' \rrbracket_{(G_B,w_B)} < \llbracket S_B \rrbracket_{(G_B,w_B)}$. Since $w_B$ is a weight function for $G_B$, we know that $w_B(u,v) \leq w_B(uv)$ for all $uv \in {B \cup C \choose 2}$. We now have that
\begin{displaymath}
\begin{array}{rcl}
0 & < & \llbracket S_B \rrbracket_{(G_B,w_B)}-\llbracket S_B' \rrbracket_{(G_B,w_B)}
\\
\\
& = & \Big(\llbracket S_C \rrbracket_{(G_B[C],w_B)}-\llbracket S_C' \rrbracket_{(G_B[C],w_B)}\Big)+
\\
& & +\Big(\sum\limits_{u \in S_C \smallsetminus S_C'}\sum\limits_{v \in B \smallsetminus S_B} (w_B(u,v)-w_B(uv))\Big)
\\
\\
& \leq & \llbracket S_C \rrbracket_{(G_B[C],w_B)}-\llbracket S_C' \rrbracket_{(G_B[C],w_B)}.
\\
\\
& = & \alpha_{A \cup S_C}-\alpha_{A \cup S_C'},
\end{array}
\end{displaymath}
and consequently,
$\alpha_{A\cup S_C'} < \alpha_{A\cup S_C}$, as we had claimed.

Now, using Proposition~\ref{prop-delete-zero}, we fix a stable set $S_A \subseteq A \cup S_C$ of $G_A$ that contains no zero-vertices of $G_A$ and satisfies $\llbracket S_A \rrbracket_{({\rm Red}[G_A,w;A \cup S_C])} = \alpha({\rm Red}[G_A,w;A \cup S_C])$. By Proposition~\ref{prop-ext}, we have that
\begin{displaymath}
\begin{array}{rcl}
\llbracket S_A \rrbracket_{(G_A,w)} & = & \llbracket S_A \rrbracket_{{\rm Red}[G_A,w;A \cup S_C]}+{\rm Ext}[G_A,w;A \cup S_C]
\\
\\
& = & \alpha({\rm Red}[G_A,w;A \cup S_C])+{\rm Ext}[G_A,w;A \cup S_C]
\\
\\
& = & \alpha_{A \cup S_C}.
\end{array}
\end{displaymath}
Next, note the following:
\begin{displaymath}
\begin{array}{rcllll}
\alpha_{A \cup S_C} & = & \llbracket S_A \rrbracket_{(G_A,w)}
\\
\\
& \leq & \llbracket S_A \rrbracket_{{\rm Red}[G_A,w;A \cup (S_A \cap C)]}+ & & & \text{by Proposition~\ref{prop-ext}}
\\
& & +{\rm Ext}[G_A,w;A \cup (S_A \cap C)]
\\
\\
& \leq & \alpha({\rm Red}[G_A,w;A \cup (S_A \cap C)])+
\\
& & +{\rm Ext}[G_A,w;A \cup (S_A \cap C)]
\\
\\
& = & \alpha_{A \cup (S_A \cap C)}.
\end{array}
\end{displaymath}
Thus, $\alpha_{A \cup S_C} \leq \alpha_{A \cup (S_A \cap C)}$. Now, recall that for all $S_C' \subsetneqq S_C$, we have that $\alpha_{A \cup S_C'} < \alpha_{A \cup S_C}$; since (by the construction of $S_A$) we have that $S_A \cap C \subseteq S_C$, this implies that $S_C = S_A \cap C$.

Set $S = S_A \cup S_B$; since $S_A \cap C = S_C = S_B \cap C$, and since $(A,B,C)$ is a cut-partition of $G$, we readily deduce that $S$ is a stable set of $G$. We now compute:
\begin{displaymath}
\begin{array}{rclllll}
k+\alpha(G_B,w_B) &= & k+\llbracket S_B \rrbracket_{(G_B,w_B)}
\\
\\
& = & k+(\alpha_{A \cup S_C}-k)+
\\
& & +\llbracket S_B \rrbracket_{(G_B,w_B)}-
\\
& & -\llbracket S_C \rrbracket_{(G_B[C],w_B)}
\\
\\
& = & \alpha_{A \cup S_C}+\llbracket S_B \rrbracket_{(G_B,w)}- & & & \text{by Proposition~\ref{differ-at-C-only}}
\\
& & -\llbracket S_C \rrbracket_{(G_B[C],w)}
\\
\\
& = & \llbracket S_A \rrbracket_{(G_A,w)}+
\\
& & +\llbracket S_B \rrbracket_{(G_B,w)}-
\\
& & -\llbracket S_C \rrbracket_{(G_B[C],w)}
\\
\\
& = & \llbracket S_A \rrbracket_{(G[A \cup C],w)}+
\\
& & +\llbracket S_B \rrbracket_{(G[B \cup C],w)}-
\\
& & -\llbracket S_C \rrbracket_{(G[C],w)}
\\
\\
& = & \llbracket S \rrbracket_{(G,w)} & & & \text{by Proposition~\ref{weight-cut-part}}
\\
\\
& \leq & \alpha(G,w).
\end{array}
\end{displaymath}
This completes the argument.
\end{proof}

\begin{lemma} \label{lemma-weights-clique-cut} Let $(G,w)$ be a weighted trigraph, let $C$ be a clique-cutset of $G$, and let $(A,B,C)$ be an associated cut-partition of $G$. Set $G_A = G[A \cup C]$ and $G_B = G[B \cup C]$. For each $C' \subseteq C$, set $\alpha_{A \cup C'} = \alpha({\rm Red}[G_A,w;A \cup C'])+{\rm Ext}[G_A,w;A \cup C']$. Define $w_B:D(G_B) \rightarrow \mathbb{N}$ by setting $w_B(c) = \alpha_{A \cup \{c\}}-\alpha_A$ for all $c \in C$, and $w_B \upharpoonright (D(G_B) \smallsetminus C) = w \upharpoonright (D(G_B) \smallsetminus C)$. Then $w_B$ is a weight function for $G_B$, and $\alpha(G,w) = \alpha_A+\alpha(G_B,w_B)$.
\end{lemma}
\begin{proof}
By Proposition~\ref{prop-alpha-nested}, we have that $w_B(c) \geq 0$ for all $c \in C$, and it follows immediately that $w_B$ is a weight function for $G_B$. Now, set $k = \alpha_A$. Using the fact that $C$ is a strong clique of $G_B$, we observe that the weight function $w_B$ for $G_B$ satisfies the hypotheses of Proposition~\ref{prop-cut-part-reduction}, and we deduce that $\alpha(G,w) = \alpha_A+\alpha(G_B,w_B)$.
\end{proof}

\begin{lemma} \label{lemma-weights-proper-2-cut} Let $(G,w)$ be a weighted trigraph and let $(A,B,C)$ be a cut-partition of $G$ such that $C$ is a stable set of size two of $G$. Set $C = \{c_1,c_2\}$. For each $X \in \{A,B\}$, let $G_X$ be the trigraph on the vertex set $X \cup C$ in which $c_1c_2$ is a semi-adjacent pair and all other adjacencies are inherited from $G[X \cup C]$. For each $C' \subseteq C$, set $\alpha_{A \cup C'} = \alpha({\rm Red}[G_A,w;A \cup C'])+{\rm Ext}[G_A,w;A \cup C']$. Define $w_B:D(G_B) \rightarrow \mathbb{N}$ as follows:
\begin{itemize}
\item $w_B(c_1) = \alpha_{A \cup C}-w(c_2)$;
\item $w_B(c_2) = w(c_2)$;
\item $w_B(c_1,c_2) = \alpha_{A \cup \{c_1\}}-\alpha_{A \cup C}+w(c_2)$;
\item $w_B(c_2,c_1) = \alpha_{A \cup \{c_2\}}-w(c_2)$;
\item $w_B(c_1c_2) = \alpha_A$;
\item $w_B \upharpoonright \Big(D(G_B) \smallsetminus D(G_B[C])\Big) = w \upharpoonright \Big(D(G_B) \smallsetminus D(G_B[C])\Big)$.
\end{itemize}
Then $w_B$ is a weight function for $G_B$, and $\alpha(G_B,w_B) = \alpha(G,w)$.
\end{lemma}
\begin{proof}
We first show that $w_B$ is a weight function for $G_B$. It suffices to show that $w_B(c_1),w_B(c_1,c_2),w_B(c_2,c_1) \geq 0$ and that $w_B(c_1,c_2),w_B(c_2,c_1) \leq w_B(c_1c_2)$, for $w_B$ clearly satisfies all the other conditions from the definition of a weight function. The fact that $w_B(c_1),w_B(c_2,c_1) \geq 0$ follows immediately from Proposition~\ref{prop-S-red-ext}. Next, Proposition~\ref{prop-alpha-nested} guarantees that $\alpha_{A \cup C} \leq \alpha_{A \cup \{c_1\}}+w(c_2)$, which immediately implies that $w_B(c_1,c_2) \geq 0$. Similarly, Proposition~\ref{prop-alpha-nested} guarantees that $\alpha_{A \cup \{c_2\}} \leq \alpha_A+w(c_2)$, which implies that $w_B(c_2,c_1) \leq w_B(c_1c_2)$. Finally, to show that $w_B(c_1,c_2) \leq w_B(c_1c_2)$, we observe that:
\begin{displaymath}
\begin{array}{rcllll}
w_B(c_1,c_2) & = & \alpha_{A \cup \{c_1\}}-\alpha_{A \cup C}+w(c_2)
\\
\\
& \leq & \Big(\alpha_A+w(c_1)\Big)-\alpha_{A \cup C}+w(c_2) & & & \text{by Proposition~\ref{prop-alpha-nested}}
\\
\\
& = & \alpha_A+\Big(w(c_1)+w(c_2)\Big)-\alpha_{A \cup C}
\\
\\
& \leq & \alpha_A & & & \text{by Proposition~\ref{prop-S-red-ext}}
\\
\\
& = & w_B(c_1c_2).
\end{array}
\end{displaymath}
This proves that $w_B$ is indeed a weight function for $G_B$.

Now, set $k = 0$. We see by inspection that $w_B$ satisfies the hypotheses of Proposition~\ref{prop-cut-part-reduction}, and we deduce that $\alpha(G_B,w_B) = \alpha(G,w)$. This completes the argument.
\end{proof}

\section{Decomposition theorem}
\label{s:dec}

In this section, we state a decomposition theorem for \{ISK4,wheel\}-free trigraphs (see Theorem~\ref{thm-decomp} below), and then we derive an ``extreme'' decomposition theorem for this class of graphs, which states (roughly) that every \{ISK4,wheel\}-free trigraph is either ``basic'' or admits a ``decomposition'' such that one of the ``blocks of decomposition'' is basic (see Theorem~\ref{alg-extreme-decomp} and Corollary~\ref{cor-extreme-decomp}). Here, we state Theorem~\ref{thm-decomp} without proof, but the interested reader can find a complete proof in~\cite{ISK4wheel-decomp}. As explained in the Introduction, the proof of Theorem~\ref{thm-decomp} closely follows the proof of the decomposition theorem for ISK4-free graphs from~\cite{MR2927414}, but the proof of our theorem is easier because we restrict ourselves to the wheel-free case. Interestingly, the fact that we work with trigraphs rather than graphs does not make the proof significantly harder.

\begin{theorem}\cite{ISK4wheel-decomp} \label{thm-decomp} Let $G$ be an \{ISK4,wheel\}-free trigraph. Then at least one of the following holds:
\begin{itemize}
\item $G$ is a series-parallel trigraph;
\item $G$ is a complete bipartite trigraph;
\item $G$ is a line trigraph;
\item $G$ admits a clique-cutset;
\item $G$ admits a stable 2-cutset.
\end{itemize}
\end{theorem}

We remark that L{\'e}v{\^e}que, Maffray, and Trotignon~\cite{MR2927414} proved a graph analogue of Theorem~\ref{thm-decomp}. Their theorem had an additional outcome, namely, that $G$ is a ``long rich square.'' In fact, long rich squares are not wheel-free, and so this outcome is 
unnecessary (see~\cite{ISK4wheel-decomp} for details). Furthermore, the last outcome of the decomposition theorem for \{ISK4,wheel\}-free graphs from~\cite{MR2927414} is that the graph admits a proper 2-cutset. In the trigraph context, we work with stable 2-cutsets instead.

Let us say that $G$ is a {\em basic trigraph} if $G$ is either a series-parallel trigraph, a complete bipartite trigraph, or a line trigraph. Note that all induced subtrigraphs of a basic trigraph are basic trigraphs.

A {\em good cut-partition} of a trigraph $G$ is a cut-partition $(A,B,C)$ of $G$ such that either
\begin{itemize}
\item $C$ is a clique-cutset of $G$ such that $|C| \leq 3$ (in this case, $(A,B,C)$ is said to be a good cut-partition of {\em type clique}), or
\item $C$ is a stable 2-cutset of $G$, and each of $G[A \cup C]$ and $G[B \cup C]$ contains a narrow path between the two vertices of $C$ (in this case, $(A,B,C)$ is said to be a good cut-partition of {\em type stable}).
\end{itemize}

\begin{proposition} \label{prop-good-cut-part} Let $G$ be an \{ISK4,wheel\}-free trigraph. Then the following are equivalent:
\begin{itemize}
\item[(a)] $G$ admits a clique-cutset or a stable 2-cutset;
\item[(b)] $G$ admits a good cut-partition.
\end{itemize}
\end{proposition}
\begin{proof}
Clearly, (b) implies (a). For the reverse, we suppose that $G$ admits a clique-cutset or a stable 2-cutset, and we show that $G$ admits a good cut-partition. If $G$ admits a clique-cutset, then let $C$ be a clique-cutset of $G$, and otherwise, let $C$ be a stable 2-cutset of $G$. Let $(A,B,C)$ be any cut-partition of $G$ induced by $C$. If $C$ is a clique-cutset, then since $G$ is ISK4-free, we see that $|C| \leq 3$, and it follows that $(A,B,C)$ is a good cut-partition of $G$ of type clique. So assume that $C$ is a stable 2-cutset. (Note that this means that $G$ admits no clique-cutset, and in particular, $G$ is connected and contains no cut-vertices.) Set $C = \{c_1,c_2\}$. We claim that $(A,B,C)$ is a good cut-partition of $G$ of type stable. To prove this, we must only show that each of $G[A \cup C]$ and $G[B \cup C]$ contains a narrow path between $c_1$ and $c_2$. By symmetry, it suffices to show that $G[A \cup C]$ contains a narrow path between $c_1$ and $c_2$. Let $A_1$ be the vertex set of a component of $G[A]$. Vertex $c_1$ must have a neighbor in $A_1$, for otherwise, $c_2$ would be a cut-vertex of $G$; similarly, vertex $c_2$ has a neighbor in $A_1$. Thus, $G[A_1 \cup \{c_1,c_2\}]$ is connected, and it follows that $G[A_1 \cup \{c_1,c_2\}]$ (and consequently $G[A \cup C]$ as well) contains a narrow path between $c_1$ and $c_2$. Thus, $(A,B,C)$ is a good cut-partition of $G$ of type stable. This completes the argument.
\end{proof}

Theorem~\ref{thm-decomp} and Proposition~\ref{prop-good-cut-part} immediately imply the following.

\begin{corollary} \label{cor-decomp} Let $G$ be an \{ISK4,wheel\}-free trigraph. Then either $G$ is a basic trigraph, or $G$ admits a good cut-partition.
\end{corollary}

Our goal for the remainder of the section is to derive an ``extreme decomposition theorem'' from Corollary~\ref{cor-decomp} (see Theorem~\ref{alg-extreme-decomp} and Corollary~\ref{cor-extreme-decomp}).

Given a good cut-partition $(A,B,C)$ of a trigraph $G$, and given $X \in \{A,B\}$, we define the {\em $X$-block} of $G$ with respect to $(A,B,C)$ as follows:
\begin{itemize}
\item if $(A,B,C)$ is of type clique, then $G_X = G[X \cup C]$;
\item if $(A,B,C)$ is of type stable, then $G_X$ is the trigraph obtained from $G[X \cup C]$ by making the two vertices of $C$ semi-adjacent.
\end{itemize}
We remark that $G_X$ is well-defined because every good cut-partition is either of type clique or of type stable, but not both. We also remark that if $(A,B,C)$ is of type stable and the two vertices of $C$ are semi-adjacent in $G$, then $G_X = G[X \cup C]$.

\begin{proposition} \label{blocks-free} Let $(A,B,C)$ be a good cut-partition of an \{ISK4,wheel\}-free trigraph $G$, and for each $X \in \{A,B\}$, let $G_X$ be the $X$-block of $G$ with respect to $(A,B,C)$. Then $G_A$ and $G_B$ are \{ISK4,wheel\}-free.
\end{proposition}
\begin{proof}
By symmetry, it suffices to show that $G_A$ is \{ISK4,wheel\}-free. We may assume that $G_A \neq G[A \cup C]$, for otherwise, we are done. It now follows from the construction of $G_A$ that $(A,B,C)$ is of type stable, and that the two vertices of $C$ (call them $c_1$ and $c_2$) are strongly anti-adjacent in $G$. Furthermore, $G_A$ is obtained from $G[A \cup C]$ by turning the strongly anti-adjacent pair $c_1c_2$ into a semi-adjacent pair. Let $\widetilde{G}_A$ be some realization of $G_A$; we must show that $\widetilde{G}_A$ is \{ISK4,wheel\}-free. If $c_1c_2$ is a non-edge of $\widetilde{G}_A$, then $\widetilde{G}_A$ is an induced subgraph of some realization of $G[A \cup C]$, and since $G$ is \{ISK4,wheel\}-free, so is $\widetilde{G}_A$. So assume that $c_1c_2$ is an edge of $\widetilde{G}_A$. Since $(A,B,C)$ is a good cut-partition of $G$ of type stable, we know that $G[B \cup C]$ contains a narrow path $P$ between $c_1$ and $c_2$. Then some realization $H$ of $G[A \cup V(P)]$ is a subdivision of $\widetilde{G}_A$. Since $G$ is \{ISK4,wheel\}-free, so is $H$. Note that every subdivision of an ISK4 is an ISK4, and that every subdivision of a wheel contains either an induced wheel or an ISK4. Thus, if $\widetilde{G}_A$ contained an ISK4 or an induced wheel, then all its subdivisions would also contain an ISK4 or an induced wheel. Since the \{ISK4,wheel\}-free graph $H$ is a subdivision of $\widetilde{G}_A$, it follows that $\widetilde{G}_A$ is an \{ISK4,wheel\}-free graph. This completes the argument.
\end{proof}

\begin{proposition} \label{alg-find-good-cut-part} There is an algorithm with the following specifications:
\begin{itemize}
\item Input: a trigraph $G$;
\item Output: exactly one of the following:
\begin{itemize}
\item a good cut-partition $(A,B,C)$ of $G$ of type clique, together with the true statement ``$(A,B,C)$ is a good cut-partition of $G$ of type clique'';
\item a good cut-partition $(A,B,C)$ of $G$ of type stable, together with the true statement ``$(A,B,C)$ is a good cut-partition of $G$ of type stable, and $G$ does not admit a good cut-partition of type clique'';
\item the true statement ``$G$ does not admit a good cut-partition'';
\end{itemize}
\item Running time: $O(n^5)$, where $n = |V(G)|$.
\end{itemize}
\end{proposition}
\begin{proof}
Let $G_f$ be the full realization of $G$; clearly, $G_f$ can be constructed in $O(n^2)$ time.
We first form a list $C_1,\dots,C_k$ of all (possibly empty) strong cliques of $G$ of size at most three; there are at most ${n \choose 0}+{n \choose 1}+{n \choose 2}+{n \choose 3}$ such cliques, and the list $C_1,\dots,C_k$ can be found in time $O(n^3)$. For each $i \in \{1,\dots,k\}$, we can determine in time $O(n^2)$ whether $C_i$ is a cutset of $G_f$; since we are testing $O(n^3)$ cliques, we can determine whether $G$ has a clique-cutset of size at most three in $O(n^5)$ time. If we determined that some $C_i$ from the list is a cutset of $G_f$ (and therefore of $G$), then we can find the components $A_1,\dots,A_t$ ($t \geq 2$) of $G \smallsetminus C_i$ in time $O(n^2)$. In this case, $(V(A_1),\bigcup_{j=2}^t V(A_j),C_i)$ is a good cut-partition of $G$ type clique, and the algorithm returns this cut-partition and stops. So assume that the algorithm determined that $G$ contains no clique-cutsets of size at most three, and consequently, $G$ admits not good cut-partition of type clique. (In particular then, $G$ is connected and contains no cut-vertices.)

We then form a list $S_1,\dots,S_\ell$ of all (not necessarily strong) stable sets of size two of $G$. There are at most ${n \choose 2}$ such stable sets, and so this list can be formed in $O(n^2)$ time. For each $i \in \{1,\dots,\ell\}$, we can determine in time $O(n^2)$ whether $S_i$ is a cutset of $G_f$; since there are $O(n^2)$ sets in our list, testing the whole list takes $O(n^4)$ time. If none of $S_1,\dots,S_\ell$ is a cutset of $G_f$, then $G$ contains no stable 2-cutsets; in this case, the returns the true statement that $G$ admits no good cut-partition and stops. So assume that the algorithm determined that some $S_i$ from the list is a cutset of $G_f$ (and therefore of $G$); clearly, $S_i$ is a stable 2-cutset of $G$. We now find the components $A_1,\dots,A_t$ ($t \geq 2$) of $G_f \smallsetminus S_i$, and using the fact that $G$ is connected and admits no cut-vertex, we deduce that $(V(A_1),\bigcup_{j=2}^t V(A_j),S_i)$ is a good cut-partition of $G$ of type stable. The algorithm now returns this cut-partition and stops.

It is clear that the algorithm is correct, and that its running time is $O(n^5)$.
\end{proof}

\begin{lemma} \label{alg-good-cut-part-rec} There is an algorithm with the following specifications:
\begin{itemize}
\item Input: a trigraph $G$ and a good cut-partition $(A,B,C)$ of $G$;
\item Output: either the true statement ``the $A$-block of $G$ with respect to $(A,B,C)$ does not admit a good cut-partition,'' or a good cut-partition $(A',B',C')$ of $G$ such that $A' \cup C' \subsetneqq A \cup C$;
\item Running time: $O(n^5)$, where $n = |V(G)|$.
\end{itemize}
\end{lemma}
\begin{proof}
We first form $G_A$, the $A$-block of $G$ with respect to $(A,B,C)$; this takes $O(n^2)$ time. We then apply the algorithm from Proposition~\ref{alg-find-good-cut-part} to $G_A$; this takes $O(n^5)$ time. If the algorithm from Proposition~\ref{alg-find-good-cut-part} returns the answer that $G_A$ admits no good cut-partition, then we are done. So assume that the algorithm from Proposition~\ref{alg-find-good-cut-part} returned a good cut-partition $(A_1,B_1,C_1)$ of $G_A$. By the construction of $G_A$, we know that $C$ is a clique of $G_A$ (indeed, $C$ is either a strong clique of size at most three of $G_A$, or a set of two semi-adjacent vertices of $G_A$), and consequently, either $C \subseteq A_1 \cup C_1$ or $C \subseteq B_1 \cup C_1$. By symmetry, we may assume that $C \subseteq B_1 \cup C_1$. Now $(A_1,B \cup B_1,C_1)$ is a cut-partition of $G$, and clearly $A_1 \cup C_1 \subsetneqq A \cup C$. The algorithm now returns the cut-partition $(A_1,B \cup B_1,C_1)$ and stops.

It is clear that the running time of the algorithm is $O(n^5)$. To show that the algorithm is correct, we must show that $(A_1,B \cup B_1,C_1)$ is a good cut-partition of $G$. If $G_A = G[A \cup C]$, or if the good cut-partition $(A_1,B_1,C_1)$ of $G_A$ is of type clique, then it is clear that $(A_1,B \cup B_1,C_1)$ is a good cut-partition of $G$, and furthermore, the good cut-partition $(A_1,B \cup B_1,C_1)$ of $G$ is of the same type (type clique or type stable) as the good cut-partition $(A_1,B_1,C_1)$ of $G_A$. So assume that $G_A \neq G[A \cup C]$ and that the good cut-partition $(A_1,B_1,C_1)$ of $G_A$ is of type stable. We now claim that $(A_1,B \cup B_1,C_1)$ is a good cut-partition of $G$ of type stable.

Since $G_A \neq G[A \cup C]$, we deduce from the construction of $G_A$ that $(A,B,C)$ is a good cut-partition of $G$ of type stable, and furthermore, that the two vertices of $C$ (call them $c$ and $c'$) are strongly anti-adjacent in $G$ and semi-adjacent in $G_A$. Since $(A_1,B_1,C_1)$ is a good cut-partition of $G_A$ of type stable, we know that $C_1$ is a stable set of $G_A$ (and consequently, a stable set of $G$) of size two; set 
$C_1 = \{c_1,c_1'\}$. Furthermore, the specifications of the algorithm from Proposition~\ref{alg-find-good-cut-part} guarantee that $G_A$ does not admit a good cut-partition of type clique (for otherwise, the algorithm from Proposition~\ref{alg-find-good-cut-part} would have returned such a cut-partition), and consequently, $G_A$ is connected and contains no cut-vertices. We also note that since $C \subseteq B_1 \cup C_1$, we have either that $G_A[A_1 \cup C_1] = G[A_1 \cup C_1]$, or that $C = C_1$ and $G_A[A_1 \cup C_1]$ is obtained from 
$G[A_1 \cup C_1]$ by turning the strongly anti-adjacent pair $cc' = c_1c_1'$ into a semi-adjacent pair.

Now, to show that $(A_1,B \cup B_1,C_1)$ is a good cut-partition of $G$ of type stable, we need only show that each of $G[A_1 \cup C_1]$ and $G[B \cup B_1 \cup C_1]$ contains a narrow path between $c_1$ and $c_1'$. Let us first show that $G[A_1 \cup C_1]$ contains a narrow path between $c_1$ and $c_1'$. Let $A_1'$ be the vertex set of some component of 
$G_A[A_1] = G[A_1]$. 
Since $G_A$ contains no cut-vertices, we know that each of $c_1$ and $c_1'$ has a neighbor in $A_1'$ in $G_A$; consequently, each of $c_1$ and $c_1'$ has a neighbor in $A_1'$ in $G$. Thus, $G[A_1' \cup \{c_1,c_1'\}]$ is connected, and it follows that $G[A_1' \cup \{c_1,c_1'\}]$ (and consequently $G[A_1 \cup C_1]$ as well) contains a narrow path between $c_1$ and $c_1'$.

It remains to show that $G[B \cup B_1 \cup C_1]$ contains a narrow path between $c_1$ and $c_1'$. Since $(A_1,B_1,C_1)$ is a good cut-partition of $G_A$ of type stable, we know that there is a narrow path $P$ between $c_1$ and $c_1'$ in $G_A[B_1 \cup C_1]$. Further, since $(A,B,C)$ is a good cut-partition of $G$ of type stable, we know that there is a narrow path $Q$ between $c$ and $c'$ in $G[B \cup C]$. Now, we know that $G_A[B_1 \cup C_1]$ is the trigraph obtained from $G[B_1 \cup C_1]$ by turning the strongly anti-adjacent pair $cc'$ into a semi-adjacent pair. Thus, if $P$ contains at most one of $c$ and $c'$, then the narrow path $P$ between $c$ and $c'$ is an induced subtrigraph of $G[B \cup B_1 \cup C_1]$, and if $P$ contains both $c$ and $c'$, then $G[V(P) \cup V(Q)]$ is a narrow path in $G[B \cup B_1 \cup C_1]$ between $c_1$ and $c_1'$ (essentially, $G[V(P) \cup V(Q)]$ is the narrow path obtained from $P$ by replacing the semi-adjacent pair $cc'$ by the narrow path $Q$). This completes the argument.
\end{proof}

\begin{lemma} \label{alg-good-cut-part-min} There exists an algorithm with the following specifications:
\begin{itemize}
\item Input: a trigraph $G$;
\item Output: exactly one of the following:
\begin{itemize}
\item the true statement ``$G$ admits no good cut-partition'';
\item a good cut-partition $(A,B,C)$ of $G$, and the true statement ``the $A$-block of $G$ with respect to $(A,B,C)$ admits no good cut-partition'';
\end{itemize}
\item Running time: $O(n^6)$, where $n = |V(G)|$.
\end{itemize}
\end{lemma}
\begin{proof}
\textbf{Step 1.} We first call the algorithm from Proposition~\ref{alg-find-good-cut-part} with input $G$; the running time of that algorithm is $O(n^5)$. If the algorithm from Proposition~\ref{alg-find-good-cut-part} returns the answer that $G$ admits no good cut-partition, then we are done. So assume that the algorithm from Proposition~\ref{alg-find-good-cut-part} returned a good cut-partition $(A,B,C)$ of $G$. We now go to Step 2.

\textbf{Step 2.} We call the algorithm from Lemma~\ref{alg-good-cut-part-rec} with input $G$ and $(A,B,C)$. If the algorithm from Lemma~\ref{alg-good-cut-part-rec} returns the answer that the $A$-block of $G$ with respect to $(A,B,C)$ does not admit a good cut-partition, then we are done. So assume that the algorithm from Lemma~\ref{alg-good-cut-part-rec} returned a good cut-partition $(A',B',C')$ of $G$ such that $A' \cup C' \subsetneqq A \cup C$. We now set $(A,B,C) := (A',B',C')$, and we go back to Step 2.

Since the size of $A \cup C$ decreases after each call of Step 2, we make at most $n$ 
recursive calls to Step 2 (and in particular, the algorithm terminates). Since the running time of the algorithm from Lemma~\ref{alg-good-cut-part-rec} is $O(n^5)$, we conclude that the running time of our algorithm is $O(n^6)$.
\end{proof}

\begin{theorem} \label{alg-extreme-decomp} There exists an algorithm with the following specifications:
\begin{itemize}
\item Input: an \{ISK4,wheel\}-free trigraph $G$;
\item Output: exactly one of the following:
\begin{itemize}
\item the true statement ``$G$ is a basic trigraph'';
\item a good cut-partition $(A,B,C)$ of $G$, and the true statement ``the $A$-block of $G$ with respect to $(A,B,C)$ is a basic trigraph'';
\end{itemize}
\item Running time: $O(n^6)$, where $n = |V(G)|$.
\end{itemize}
\end{theorem}
\begin{proof}
We call the algorithm from Lemma~\ref{alg-good-cut-part-min} with input $G$. If that algorithm returns the answer that $G$ admits no good cut-partition, then our algorithm returns the answer that $G$ is a basic trigraph and stops. On the other hand, if the algorithm from Lemma~\ref{alg-good-cut-part-min} returns a good cut-partition $(A,B,C)$ of $G$ and the statement that the $A$-block of $G$ with respect to $(A,B,C)$ admits no good cut-partition, then our algorithm stops and returns the good cut-partition $(A,B,C)$ of $G$ and the statement that the $A$-block of $G$ with respect to $(A,B,C)$ is a basic trigraph.

Since the running time of the algorithm from Lemma~\ref{alg-good-cut-part-min} is $O(n^6)$, the running time of our algorithm is also $O(n^6)$. The correctness of our algorithm follows immediately from Corollary~\ref{cor-decomp} and Proposition~\ref{blocks-free}.
\end{proof}

The following ``extreme decomposition theorem'' for \{ISK4,wheel\}-free trigraphs is an immediate corollary of Theorem~\ref{alg-extreme-decomp}.

\begin{corollary} \label{cor-extreme-decomp} Let $G$ be an \{ISK4,wheel\}-free trigraph. Then either $G$ is a basic trigraph, or $G$ admits a good cut-partition $(A,B,C)$ such that the $A$-block of $G$ with respect to $(A,B,C)$ is a basic trigraph.
\end{corollary}

\section{A stability preserving transformation}
\label{s:op}

We now describe a transformation on a weighted trigraph that preserves the stability number, while decreasing the number
of semi-adjacent pairs. It is based on the {\it gem}, a graph $G$ with five vertices such that one of them, say $v$, is adjacent to all the others and $G \smallsetminus v$ is isomorphic to the four-vertex path $P_4$.

\begin{sloppypar}
Let $(G,w)$ be a weighted trigraph and let $uv$ be a semi-adjacent pair in $G$.
The weighted trigraph obtained from $(G,w)$ by {\it replacing $uv$ with a gem} is the weighted
trigraph $(G',w')$ defined as follows:
\begin{itemize}
 \item The vertex set is $V(G') = V(G)\cup \{x_{uv},x_{v,u},x_{u,v}\}$, where $x_{uv},x_{v,u},x_{u,v}$ are pairwise distinct and do not belong to $V(G)$.
 \item The adjacency function is $\theta_{G'}:{V(G')\choose 2}\to \{-1,0,1\}$, defined as follows:
 \begin{itemize}
  \item $\theta_{G'}\upharpoonright ({V(G)\choose 2}\smallsetminus\{uv\})=
 \theta_{G}\upharpoonright ({V(G)\choose 2}\smallsetminus\{uv\})\,$
  \item $\theta_{G'}(e) = 1$ for all $e\in \{ux_{v,u},\, x_{v,u}x_{u,v},\, x_{u,v}v,\, x_{uv}u,\, x_{uv}x_{v,u},\,$ $x_{uv}x_{u,v},\, x_{uv}v\}$, 
 \item  $\theta_{G'}(e) = -1$ for all other $e\in {V(G')\choose 2}$.
 \end{itemize}
 (In particular, $G'[u,x_{v,u},x_{u,v},v,x_{uv}]$ is a graph isomorphic to a gem.)
  \item The weight function $w':D(G')\to\mathbb{N}$
  is defined as follows:
\begin{itemize}
 \item $w'\upharpoonright (D(G)\smallsetminus \{uv,(v,u),(u,v)\})=
 w\upharpoonright (D(G)\smallsetminus \{uv,(v,u),(u,v)\})$,
 \item $w'(x_{uv}) = w(uv)$,
 $w'(x_{v,u}) = w(v,u)$,
 $w'(x_{u,v}) = w(u,v)$, and
 \item  $w'(p) = 0$ for all other $p\in D(G')\,.$
\end{itemize}
\end{itemize}
\end{sloppypar}

It is immediate to see that $(G',w')$ is indeed a weighted trigraph, that is, that $w'$ is a weight function of $G'$.
The importance of the above transformation stems from the fact that it preserves the stability number, a fact
which we now prove.

\begin{proposition}\label{prop:gem}
Let $uv$ be a semi-adjacent pair in a weighted trigraph $(G,w)$ and let $(G',w')$
be the weighted trigraph obtained from $(G,w)$ by replacing $uv$ with a gem.
Then, $\alpha(G',w') = \alpha(G,w)$.
\end{proposition}

\begin{proof}
Let $x_{uv},x_{v,u},x_{u,v}$ be the three vertices in $V(G')\smallsetminus V(G)$ labeled as in the definition of
the operation of replacing a semi-adjacent pair with a gem.

We split the proof of the equality $\alpha(G',w') = \alpha(G,w)$ into two parts.
First, we show that $\alpha(G,w) \leq \alpha(G',w')$.
Let $S\in V(G)$ be a stable set of $G$ such that
$\llbracket S\rrbracket_{(G,w)} = \alpha(G,w)$.
We consider three cases depending on the number of vertices in $S\cap \{u,v\}$.
In each case, we exhibit a stable set $S'$ of $G'$ such that $\llbracket S' \rrbracket_{(G',w')}
= \llbracket S \rrbracket_{(G,w)}$. This is enough, for then we obtain that
$\alpha(G,w) = \llbracket S\rrbracket_{(G,w)} = \llbracket S'\rrbracket_{(G',w')} \leq \alpha(G',w')$,
which is what we need.

If $|S\cap \{u,v\}| = 0$, then the set $S' = S\cup \{x_{uv}\}$ is a stable set of $G'$.
Its weight with respect to $(G',w')$ is
\begin{displaymath}
\begin{array}{rcl}
\llbracket S'\rrbracket_{(G',w')} &=& \sum\limits_{x\in S'}w'(x)+\sum\limits_{x\in S'}\sum\limits_{y\in V(G)\smallsetminus S'}w'(x,y)+\sum\limits_{xy\in {V(G')\smallsetminus S'\choose 2}}w'(xy)
\\
\\
&=& \left(\sum\limits_{x\in S}w(x)+w'(x_{uv})\right)+\sum\limits_{x\in S}\sum\limits_{y\in V(G)\smallsetminus S}w(x,y) +
\\
&& + \left(\sum\limits_{xy\in {V(G)\smallsetminus S\choose 2}}w(xy) -w(uv)\right)
\\
\\
&=& \sum\limits_{x\in S}w(x)+\sum\limits_{x\in S}\sum\limits_{y\in V(G)\smallsetminus S}w(x,y)+\sum\limits_{xy\in {V(G)\smallsetminus S\choose 2}}w(xy)
\\
\\
&=& \llbracket S\rrbracket_{(G,w)}.
\end{array}
\end{displaymath}

If $|S\cap \{u,v\}| = 1$, then we may assume without loss of generality that $S\cap \{u,v\} = \{u\}$.
The set $S' = S\cup \{x_{u,v}\}$ is a stable set of $G'$.
Its weight with respect to $(G',w')$ is
\begin{displaymath}
\begin{array}{rcl}
\llbracket S'\rrbracket_{(G',w')} & = & \sum\limits_{x\in S'}w'(x)+\sum\limits_{x\in S'}\sum\limits_{y\in V(G)\smallsetminus S'}w'(x,y)+\sum\limits_{xy\in {V(G')\smallsetminus S'\choose 2}}w'(xy)
\\
\\
&=& \left(\sum\limits_{x\in S}w(x)+w'(x_{u,v})\right)+ 
\\
& & +\left(\sum\limits_{x\in S}\sum\limits_{y\in V(G)\smallsetminus S}w(x,y)-w(u,v)\right)+\sum\limits_{xy\in {V(G)\smallsetminus S\choose 2}}w(xy)
\\
\\
&=& \sum\limits_{x\in S}w(x)+\sum\limits_{x\in S}\sum\limits_{y\in V(G)\smallsetminus S}w(x,y)+\sum\limits_{xy\in {V(G)\smallsetminus S\choose 2}}w(xy)
\\
\\
&=&
\llbracket S\rrbracket_{(G,w)}.
\end{array}
\end{displaymath}

Finally, suppose that $|S\cap \{u,v\}| = 2$, that is, $\{u,v\}\subseteq S$.
In this case, $S'=S$ itself is a stable set of $G'$.
It is immediate to verify that its weight with respect to $(G',w')$ is the same as its weight with respect to $(G,w)$.

We now prove the reverse inequality, that is, we show that $\alpha(G',w') \leq \alpha(G,w)$. Let $S' \subseteq V(G')$ be a stable set of $G'$ such that $\llbracket S'\rrbracket_{(G',w')} = \alpha(G',w')$.

Up to symmetry, it suffices to analyze four cases depending on the intersection of $S'$ with the vertex set of the gem, that is, depending on the 
set $X = S'\cap \{u,v,x_{uv},x_{v,u},x_{u,v}\}$. These four cases are:
$$X\in \{\{x_{uv}\}, \{x_{v,u}\},\{u,x_{u,v}\}, \{u,v\}\}\,.$$
Indeed, if $X = \emptyset$, then we can replace $S'$ with $S'\cup\{x_{uv}\}$ to obtain a set with
$\llbracket S'\cup\{x_{uv}\}\rrbracket_{(G',w')} \ge \llbracket S'\rrbracket_{(G',w')}$.
If $X = \{u\}$, then we can replace $S'$ with $S'\cup\{x_{u,v}\}$ to obtain a set with
$\llbracket S'\cup\{x_{u,v}\}\rrbracket_{(G',w')} \ge \llbracket S'\rrbracket_{(G',w')}$.
Each of the remaining cases for $X$ either results in a non-stable set, or is symmetric to one of the four cases above. In each case, we
exhibit a stable set $S$ of $G$ such that $\llbracket S\rrbracket_{(G,w)} = \llbracket S'\rrbracket_{(G',w')}$. This is enough because we
then obtain $\alpha(G',w') = \llbracket S'\rrbracket_{(G',w')} = \llbracket S\rrbracket_{(G,w)} \leq \alpha(G,w)$, which is what we need.

\medskip
{\bf Case 1.} $X = \{x_{uv}\}$.

The set $S = S'\smallsetminus\{x_{uv}\}$ is a stable set of $G$.
Its weight with respect to $(G,w)$ is
\begin{displaymath}
\begin{array}{rcl}
\llbracket S\rrbracket_{(G,w)} &=& \sum\limits_{x\in S}w(x)+\sum\limits_{x\in S}\sum\limits_{y\in V(G)\smallsetminus S}w(x,y)+\sum\limits_{xy\in {V(G)\smallsetminus S\choose 2}}w(xy)
\\
\\
&=& \left(\sum\limits_{x\in S'}w'(x)-w'(x_{uv})\right)+\sum\limits_{x\in S'}\sum\limits_{y\in V(G')\smallsetminus S'}w'(x,y) +\\
&& + \left(\sum\limits_{xy\in {V(G')\smallsetminus S'\choose 2}}w'(xy)+w(uv)\right)
\\
\\
&=& \sum\limits_{x\in S'}w'(x)+\sum\limits_{x\in S'}\sum\limits_{y\in V(G')\smallsetminus S'}w'(x,y)+\sum\limits_{xy\in {V(G')\smallsetminus S'\choose 2}}w'(xy)
\\
\\
&=&
\llbracket S'\rrbracket_{(G',w')}.
\end{array}
\end{displaymath}

\medskip
{\bf Case 2.} $X = \{x_{v,u}\}$.

In this case, the set $S'' = (S'\smallsetminus \{x_{v,u}\})\cup\{x_{uv}\}$ is also a stable set of $G'$.
Since $w(uv)\ge w(v,u)$, and since neither $x_{uv}$ nor $x_{v,u}$ is an endpoint of a semi-adjacent pair of $G'$, we have that 
\begin{displaymath}
\begin{array}{rcl}
\llbracket S''\rrbracket_{(G',w')}&=&\llbracket S'\rrbracket_{(G',w')}+ w'(x_{uv})-w'(x_{v,u})
\\
\\
&=& \llbracket S'\rrbracket_{(G',w')}+ w(uv)-w(v,u)
\\
\\
&\ge& \llbracket S'\rrbracket_{(G',w')}.
\end{array}
\end{displaymath}
This implies that
$\alpha(G',w') = \llbracket S'\rrbracket_{(G',w')} \le
\llbracket S''\rrbracket_{(G',w')}\le \alpha(G',w')$, and consequently
$\llbracket S''\rrbracket_{(G',w')}= \alpha(G',w')$. Therefore, this case reduces to Case 1.

\medskip
{\bf Case 3.} $X = \{u,x_{u,v}\}$.

The set $S = S'\smallsetminus\{x_{u,v}\}$ is a stable set of $G$.
Its weight with respect to $(G,w)$ is
\begin{displaymath}
\begin{array}{rcl}
\llbracket S\rrbracket_{(G,w)}& = & \sum\limits_{x\in S}w(x)+\sum\limits_{x\in S}\sum\limits_{y\in V(G)\smallsetminus S}w(x,y) + \sum\limits_{xy\in {V(G)\smallsetminus S\choose 2}}w(xy)
\\
\\
& =& \left(\sum\limits_{x\in S'}w'(x)-w'(x_{u,v})\right)+ 
\\
& & +\left(\sum\limits_{x\in S'}\sum\limits_{y\in V(G')\smallsetminus S'}w'(x,y)+w(u,v)\right)+\sum\limits_{xy\in {V(G')\smallsetminus S'\choose 2}}w'(xy)
\\
\\
& = & \sum\limits_{x\in S'}w'(x)+\sum\limits_{x\in S'}\sum\limits_{y\in V(G')\smallsetminus S'}w'(x,y) + \sum\limits_{xy\in {V(G')\smallsetminus S'\choose 2}}w'(xy)
\\
\\
& = & \llbracket S'\rrbracket_{(G',w')}. 
\end{array}
\end{displaymath}

\medskip
{\bf Case 4.} $X = \{u,v\}$.

The set $S = S'$ itself is a stable set of $G$. It is immediate to verify that its weight with respect to $(G,w)$ is the same as its weight with respect to $(G',w')$.

This completes the argument.
\end{proof}

\section{Computing the stability number of basic weighted trigraphs}
\label{s:comp}

We remind the reader that a {\em basic trigraph} is a trigraph $G$ that is either a series-parallel trigraph, a complete bipartite trigraph, or a line trigraph.

\begin{theorem}\label{thm:alpha-basic}
There exists an algorithm with the following specifications:
\begin{itemize}
 \item Input: a weighted basic trigraph $(G,w)$;
 \item Output: $\alpha(G,w)$;
 \item Running time: $O(n^4\log n)$ where $n = |V(G)|$.
\end{itemize}
\end{theorem}

\begin{proof}
Let $(G,w)$ be a weighted basic trigraph. Then, $G$ is either (i) a series-parallel trigraph, (ii) a complete bipartite trigraph, or
(iii) a line trigraph.

Testing (i) can be done by computing in $O(n^2)$ time the full realization $G_f$ of $G$, and testing whether $G_f$ is series-parallel, which can be done in time $O(|V(G_f)|+|E(G_f)|) = O(n^2)$~\cite{MR652904}.

Testing (ii) can be done in time $O(n^2)$ by first testing if $G$ is a graph (that is, if its adjacency function only takes values $1$ and $-1$), and then testing in $O(n^2)$ time (for example, using breadth-first search) if $G$ is a complete bipartite graph.

Thus, it can be determined in time $O(n^2)$ whether (i), (ii), or neither of these two cases occurs. If neither (i) nor (ii) occurs, then (iii) must occur.

We now discuss how to compute the stability number of $(G,w)$ in each of the three cases.

\medskip
{\bf Case 1.} $G$ is a series-parallel trigraph.

Let $(G',w')$ be the weighted trigraph obtained from $(G,w)$ by replacing each semi-adjacent pair of $G$ (in any order) with a gem.
Since each replacement of a semi-adjacent pair with a gem removes one semi-adjacent pair and produces no new ones,
the resulting trigraph $G'$ has no semi-adjacent pairs, that is, it is a graph. Clearly, $|V(G')| = O(n^2)$.
Moreover, since $G'$ has exactly one edge for each strongly adjacent pair of $G$, exactly seven edges for each semi-adjacent pair of $G$,
and no other edges, we also have that $|E(G')| = O(n^2)$.

Since $G$ is a series-parallel trigraph, its full realization $G_f$ is a series-parallel graph.
Note that $G'$ is isomorphic to the graph $G''$ obtained from $G_f$ by replacing each edge $uv\in E(G_f)$ that forms a semi-adjacent pair in $G$
with a gem with vertex set
$\{u,x_{v,u},x_{u,v},v,x_{uv}\}$ and edge set
$$\{ux_{v,u},\,x_{v,u}x_{u,v},\,x_{u,v}v,\,x_{uv}u,\,x_{uv}x_{v,u},\,x_{uv}x_{u,v},\,x_{uv}v\}.$$

For a graph $H$, let us denote by ${\rm tw}(H)$ its treewidth.
We claim that the treewidth of $G''$ (and consequently that of $G'$) is at most three.
To this end, it suffices to prove the following.

\begin{sloppypar}
{\it Claim:
Let $H$ be a graph and let $H_1$ be a graph obtained from $H$ by replacing an edge $uv\in E(H)$ with a gem
with vertex set $\{u,x_{v,u},x_{u,v},v,x_{uv}\}$ and edge set
$\{ux_{v,u},x_{v,u}x_{u,v},x_{u,v}v,x_{uv}u,x_{uv}x_{v,u},x_{uv}x_{u,v},x_{uv}v\}$.
Then, ${\rm tw}(H_1) \le \max\{{\rm tw}(H),3\}$.}
\end{sloppypar}

\medskip
This is indeed enough. Since series-parallel graphs are of treewidth at most two~\cite{MR1686154},
 $G_f$ is of treewidth at most two. Applying the claim
 repeatedly to each of the graphs in the sequence of graphs transforming
 $G_f$ to $G''$ (by replacing one edge at a time with a gem) implies that
$\text{tw}(G') \le \max\{\text{tw}(G_f),3\} = 3$.

\medskip
{\it Proof of Claim:}
Recall that a graph $K=(V,E)$ is {\it chordal} if every cycle in it of length at least four has a chord, and that $\omega(K)$ denotes
the {\it clique number} of $K$, that is, the maximum size of a clique in $K$.
Moreover, the treewidth of $K$ equals the minimum value of $\omega(K')-1$ over all chordal graphs
of the form $K' = (V,E')$ where $E\subseteq E'$ (see, e.g.,~\cite[Theorem 11.1.4]{MR1686154}).

Let $H'$ be a chordal supergraph of $H$ such that $\text{tw}(H) = \omega(H')-1$.
Then, the graph $H_1'$ defined as $V(H_1') = V(H_1)$ and $E(H_1') = E(H')\cup E(H_1)\cup \{uv, ux_{u,v}\}$ is
a chordal supergraph of $H_1$ with $\omega(H_1') = \max\{\omega(H'),4\}$.
Therefore, $\text{tw}(H_1)\le \omega(H_1')-1 = \max\{\omega(H')-1,3\} = \max\{\text{tw}(H),3\}$.\qed

\medskip
We have shown that the treewidth of $G'$ is at most three. It follows that
the stability number of $(G',w')$, and hence that of $(G,w)$,
can be computed in time $O(|V(G')|) = O(n^2)$, e.g.,
by first computing a tree-decomposition of $G'$ of width at most three~\cite{MR1417901} and then applying
a dynamic programming algorithm along the tree decomposition~\cite{MR985145}.

\medskip
{\bf Case 2.} $G$ is a complete bipartite trigraph.

In this case, $G$ is a graph and all nonzero weights $w(p)>0$ for $p\in D(G)$ appear on its vertices.
Thus, if $(A,B)$ is a bipartition of $G$, we have that $\alpha(G,w) = \max\bigg\{\sum\limits_{a \in A}w(a),\sum\limits_{b \in B}w(b)\bigg\}$. It follows that in this case the stability number can be computed in time $O(n)$
(to compute $A$ and $B$, choose $v\in V(G)$ arbitrarily, and take $A = N(v)$ and $B = V(G)\smallsetminus A$, where $N(v)$ is the set of all neighbors of $v$ in $G$).

\medskip
{\bf Case 3.} $G$ is a line trigraph.

We apply a transformation similar to the one from Case 1. Namely, let $(G',w')$ be the weighted trigraph obtained from $(G,w)$ by replacing each semi-adjacent pair of $G$ (in any order) with a gem. Again, we have that the resulting trigraph $G'$ has no semi-adjacent pairs, that is, that $G'$ is a graph,
and that $|V(G')| = |E(G')| = O(n^2)$.

\begin{sloppypar}
{\it Claim:
Let $(H,w)$ be a weighted line trigraph, let $uv$ be a semi-adjacent pair in $H$, and let $(H',w')$ be the trigraph obtained from $(H,w)$ by replacing $uv$ with a gem. Then, $H'$ is also a line trigraph.}
\end{sloppypar}

{\it Proof of Claim:}
Suppose that $H$ is a line trigraph of a graph $K$. This means that the full realization $H_f$ of $H$ is the line graph of $K$ and all the triangles of $H$ are strong. Vertices $u$ and $v$ are adjacent in $H_f$, and so they correspond to a pair of adjacent edges, say
$ab$ and $bc$, respectively, in $K$. Since every triangle in $H$ is strong, the edge $uv\in E(H_f)$ is not part of any triangle in $H_f$.
This implies that $b$ is a vertex of degree two in $K$, and $ac\not\in E(K)$.
Let $K'$ be the graph defined as follows: $V(K') = (V(K)\smallsetminus\{b\})\cup\{d,e,f\}$ and $E(K') =
(E(K)\smallsetminus\{ab,bc\})\cup\{ad,de,ec,df,ef\}$.
Then, the line graph of $K'$ is isomorphic to the full realization of $H'$.
Moreover, all the triangles of $H'$ are strong. Therefore, $H'$ is a line trigraph.\qed

\medskip
Applying the above claim repeatedly to each of the trigraphs in the sequence of trigraphs transforming
$(G,w)$ to $(G',w')$ implies that $G'$ is a line trigraph.
Since $G'$ is in fact a graph, it is a line graph.
Since the operation of replacing a semi-adjacent pair with a gem preserves the stability number (by Proposition~\ref{prop:gem}),
we have that $\alpha(G,w) = \alpha(G',w')$.

It is therefore enough to compute the stability number of the weighted line graph $(G',w')$. This can be done as follows. First, compute
a graph $H'$ such that $G' = L(H')$; this can be done in time $O(|V(G')|+|E(G')|) = O(n^2)$~\cite{MR0424435}.
Second, solve the instance of the maximum weight matching problem on $H'$ with edge weights corresponding to vertex weights in $G'$. This can be done in time
$O(|V(H')|(|E(H')|+|V(H')|\log |V(H')|))$ using the algorithm by Gabow~\cite{DBLP:conf/soda/Gabow90}.

The time complexity of the whole algorithm in Case 3 is dominated by the term
$O(|V(H')|(|E(H')|+|V(H')|\log |V(H')|))$, which, since
$|V(H')| = O(|V(G')|) = O(n^2)$ and $|E(H')| = O(|V(G')|) = O(n^2)$, is of order
$O(n^4\log n)$. This completes the description of the algorithm for Case 3.

\medskip
The running time $O(n^4\log n)$ of Case 3 dominates
the running time of each of the other steps of the algorithm. This completes the proof.
\end{proof}

\section{Computing the stability number of \{ISK4,wheel\}-free weighted trigraphs}
\label{s:compClass}

\begin{sloppypar}
We now derive the main result of the paper: a polynomial-time algorithm that finds the stability number of a weighted \{ISK4,wheel\}-free trigraph. We remark that since every weighted \{ISK4,wheel\}-free graph (with non-negative integer weights) is a weighted \{ISK4,wheel\}-free trigraph, this algorithm can be used to compute the stability number of a weighted \{ISK4,wheel\}-free graph (and this is, in fact, the main purpose of our algorithm). 
\end{sloppypar}

\begin{theorem}\label{thm:main}
There exists an algorithm with the following specifications:
\begin{itemize}
 \item Input: a weighted \{ISK4,wheel\}-free trigraph $(G,w)$;
 \item Output: $\alpha(G,w)$;
 \item Running time: $O(n^7)$ where $n = |V(G)|$.
\end{itemize}
\end{theorem}

\begin{proof}
Let $(G,w)$ be the input \{ISK4,wheel\}-free trigraph with $n = |V(G)|$. We first call the $O(n^6)$ time algorithm from Theorem~\ref{alg-extreme-decomp} with input $G$. This algorithm either returns the statement that $G$ is a basic trigraph, or returns a good cut-partition $(A,B,C)$ of $G$ such that the $A$-block of $G$ with respect to $(A,B,C)$ is a basic trigraph.

If the algorithm returns the statement that $G$ is a basic trigraph, then we apply Theorem~\ref{thm:alpha-basic} and compute $\alpha(G,w)$ in time $O(n^4\log n)$.

Suppose now that the algorithm returned a good cut-partition $(A,B,C)$ of $G$ such that the $A$-block of $G$ with respect to
$(A,B,C)$ is a basic trigraph. For $X\in \{A,B\}$, let $G_X$ be the $X$-block of $G$ with respect to $(A,B,C)$. Since $(A,B,C)$ is a good cut-partition of $G$, we know that $|C| \leq 3$, and so it can be determined in $O(1)$ time whether $(A,B,C)$ is of type clique or of type stable. We now analyze the two cases.

\medskip
{\bf Case 1.} The good cut-partition $(A,B,C)$ of $G$ is of type clique.

In this case, we have $G_X = G[X\cup C]$ for $X\in \{A,B\}$.
For each of the $2^{|C|} = O(1)$ sets of the form $C'\subseteq C$, compute the weighted trigraph ${\rm Red}[G_A,w;A\cup C']$ and the quantity $\text{Ext}[G_A,w;A\cup C']$. By Proposition~\ref{prop-ext-alg}, this can be done in time $O(n^2)$.

Note that each of the reductions ${\rm Red}[G_A,w;A\cup C']$ is a weighted
induced subtrigraph of a basic trigraph (namely, $G_A$ with an appropriate weight function), and so
${\rm Red}[G_A,w;A\cup C']$ is a basic trigraph.
Therefore, by Theorem~\ref{thm:alpha-basic}, for each $C'\subseteq C$, the stability number of
${\rm Red}[G_A,w;A\cup C']$ can be computed in time
 $O(n^4\log n)$.
For each $C'\subseteq C$, set $$\alpha_{A\cup C'} = \alpha({\rm Red}[G_A,w;A\cup C'])+\text{Ext}[G_A,w;A\cup C'].$$

Define $w_B:D(G_B) \to \mathbb{N}$ by setting $w_B(c) = \alpha_{A\cup \{c\}}-\alpha_A$ for all $c\in C$, and
$w_B\upharpoonright (D(G_B)\smallsetminus C) = w\upharpoonright (D(G_B)\smallsetminus C)$.
By Lemma~\ref{lemma-weights-clique-cut}, $w_B$ is a weight function for $G_B$.
Call the algorithm recursively on the weighted trigraph $(G_B,w_B)$ to compute $\alpha(G_B,w_B)$.
By Lemma~\ref{lemma-weights-clique-cut}, the stability number $(G,w)$ is given by $\alpha(G,w) = \alpha_A+\alpha(G_B,w_B)$.

\medskip
{\bf Case 2.} The good cut-partition $(A,B,C)$ of $G$ is of type stable.

In this case, $C$ is a stable set of $G$ of size two, and $G_X$ (for $X\in \{A,B\}$) is the trigraph obtained from $G[X\cup C]$ by making the two vertices of $C$ semi-adjacent. We proceed similarly as in Case 1, using Lemma~\ref{lemma-weights-proper-2-cut} instead of Lemma~\ref{lemma-weights-clique-cut}. For each of the $2^{|C|} = O(1)$ sets of the form $C'\subseteq C$, we use Proposition~\ref{prop-ext-alg} and Theorem~\ref{thm:alpha-basic} to compute the trigraph ${\rm Red}[G_A,w;A \cup C']$ and the value $\alpha_{A\cup C'} = \alpha({\rm Red}[G_A,w;A\cup C'])+\text{Ext}[G_A,w;A\cup C']$. We then compute in $O(n^2)$ time the mapping $w_B:D(G_B)\to \mathbb{N}$ defined as in Lemma~\ref{lemma-weights-proper-2-cut}, and we call the algorithm recursively on the weighted trigraph $(G_B,w_B)$. By Lemma~\ref{lemma-weights-proper-2-cut}, the stability number of $(G,w)$ is given by $\alpha(G,w) = \alpha(G_B,w_B)$. This completes the description of the algorithm for Case~2.

\medskip
As there are at most $n-1$ recursive calls and the remaining computations take $O(n^6)$ time,
the overall running time of the algorithm is $O(n^7)$.
\end{proof}

As an immediate corollary of Theorem~\ref{thm:main} and Proposition~\ref{prop:number-to-set}, we obtain the following result.

\begin{corollary}\label{thm:main-graphs}
There exists an algorithm with the following specifications:
\begin{itemize}
 \item Input: a weighted \{ISK4,wheel\}-free graph $(G,w)$ with non-negative integer weights;
 \item Output: a maximum weight stable set $S$ of $(G,w)$;
 \item Running time: $O(n^8)$ where $n = |V(G)|$.
\end{itemize}
\end{corollary}

We remark that the algorithm from Corollary~\ref{thm:main-graphs} cannot readily be generalized to trigraphs. One reason for this is that in the graph case, one can always find a maximum weight stable set that is also an inclusion-wise maximal stable set (and this fact is implicitly used in the proof of Proposition~\ref{prop:number-to-set}), whereas this is not the case for trigraphs. We believe that one could use techniques similar to the ones from Section~\ref{s:swi} in order to generalize Corollary~\ref{thm:main-graphs} to trigraphs. However, our main interest here is in graphs, and we used trigraphs only as a tool for obtaining algorithms for graphs; for this reason, we did not attempt to construct an algorithm for trigraphs analogous to the one given by Corollary~\ref{thm:main-graphs}.

\section{Bipartite trigraphs}
\label{s:bip}

As stated in the Introduction, computing the stability number of a weighted bipartite trigraph is NP-hard. We now prove this result.

\begin{theorem} \label{Bip-NP} The problem of computing the stability number of a weighted bipartite trigraph is NP-hard.
\end{theorem}

\begin{proof}
Suppose there is a polynomial-time algorithm $\cal A$ for the problem. We prove the theorem by using $\cal A$ as a subroutine to compute the stability number of a general graph in polynomial time.

Let $H$ be an arbitrary graph, and let $n = |V(H)|$ and $m = |E(H)|$. The idea is as follows. We build a bipartite trigraph $G$ by first subdividing each edge of $H$ once, and then turning all edges of the resulting graph into semi-adjacent pairs. (Thus, $|V(G)| = n+m$.) We construct a weight function $w$ for $G$ such that $\alpha(G,w) = \alpha(H)+2m$. We can use $\cal A$ to find $\alpha(G,w)$, and because $\alpha(G,w) = \alpha(H)+2m$, we deduce that $\alpha(H)$ can be found in polynomial time. Now, describing the weight function $w$ is bit complicated because if $uv$ is an edge of $H$, the weights assigned to the two semi-adjacent pairs of $G$ that correspond to $uv$ are not symmetric between $u$ and $v$. So, in order to properly define the weight function $w$, we must first introduce some more notation.

First, let $\vec{H} = (V(\vec{H}),A(\vec{H}))$ be any orientation of $H$ (in other words, $\vec{H}$ is a digraph that satisfies $V(\vec{H}) = V(H)$, for each edge $uv \in E(H)$, exactly one of the arcs $\vec{uv}$ and $\vec{vu}$ belongs to $A(\vec{H})$, and $A(\vec{H})$ contains no other arcs). For each $\vec{uv} \in A(\vec{H})$, we introduce a new vertex $x_{\vec{uv}}$, and we set $X = \{x_{\vec{uv}} \mid \vec{uv} \in A(\vec{H})\}$. We now let $G$ be the bipartite trigraph with bipartition $(X,V(H))$ in which for all $\vec{uv} \in A(H)$, vertex $x_{\vec{uv}}$ is semi-adjacent to $u$ and $v$ and strongly anti-adjacent to all other vertices of $V(H)$. (Thus, each arc $\vec{uv}$ of $\vec{H}$ effectively gets replaced by a narrow path $u-x_{\vec{uv}}-v$.) We remark that $G$ contains no strongly adjacent pairs, and so all subsets of $V(G)$ are stable sets of $G$.

We now define a function $w:D(G) \rightarrow \mathbb{N}$ as follows:
\begin{itemize}
\item $w(v) = 1$ for all $v \in V(G)$;
\item $w(ux_{\vec{uv}}) = w(x_{\vec{uv}}v) = w(x_{\vec{uv}},v) = w(v,x_{\vec{uv}}) = 1$ for all $\vec{uv} \in A(\vec{H})$;
\item $w(e) = 0$ for all other $e \in D(G)$.
\end{itemize}
Clearly, $w$ is a weight function for $G$, and by assumption, we can find $\alpha(G,w)$ in polynomial time. (Since $|V(G)| = n+m$, the running time is in fact polynomial in $n$.) We claim that $\alpha(G,w) = \alpha(H)+2m$. This is enough, for then we can clearly compute $\alpha(H)$ in polynomial time.

We now need some more notation. For each $\vec{uv} \in A(\vec{H})$ and $S \subseteq V(G)$, set
\begin{displaymath}
\begin{array}{rcl}
{\rm cont}({\vec{uv}};S) & = & \llbracket S \cap \{u,x_{\vec{uv}},v\} \rrbracket_{(G[u,x_{\vec{uv}},v],w)}-\sum\limits_{x \in S \cap \{u,v\}} w(x).
\end{array}
\end{displaymath}
We refer to ${\rm cont}({\vec{uv}};S)$ as the {\em contribution} of the arc $\vec{uv}$ to the weight of $S$ with respect to $(G,w)$. Clearly, for all $S \subseteq V(G)$, we have that
\begin{displaymath}
\begin{array}{rcl}
\llbracket S \rrbracket_{(G,w)} & = & \sum\limits_{x \in S \cap V(H)} w(x)+\sum\limits_{\vec{uv} \in A(\vec{H})} {\rm cont}({\vec{uv}};S)
\\
\\
& = & |S \cap V(H)|+\sum\limits_{\vec{uv} \in A(\vec{H})} {\rm cont}({\vec{uv}};S).
\end{array}
\end{displaymath}
Furthermore, we see by inspection that for all $\vec{uv} \in A(H)$ and $S \subseteq V(G)$, we have that
\begin{displaymath}
\begin{array}{lll}
{\rm cont}({\vec{uv}};S) & = &
\left\{\begin{array}{lll}
1 & \text{if} & \text{either $S \cap \{u,x_{\vec{uv}},v\} = \{u,x_{\vec{uv}},v\}$}
\\
& & \text{or $S \cap \{u,x_{\vec{uv}},v\} = \{x_{\vec{uv}},v\}$}
\\
& & \text{or $S \cap \{u,x_{\vec{uv}},v\} = \{u,v\}$}
\\
& & \text{or $S \cap \{u,x_{\vec{uv}},v\} = \{u\}$}
\\
\\
2 & \text{if} & \text{either $S \cap \{u,x_{\vec{uv}},v\} = \{u,x_{\vec{uv}}\}$}
\\
& & \text{or $S \cap \{u,x_{\vec{uv}},v\} = \{v\}$}
\\
& & \text{or $S \cap \{u,x_{\vec{uv}},v\} = \{x_{\vec{uv}}\}$}
\\
& & \text{or $S \cap \{u,x_{\vec{uv}},v\} = \emptyset$}
\end{array}\right.
\end{array}
\end{displaymath}
In particular then, $1 \leq {\rm cont}(\vec{uv};S) \leq 2$ for all $\vec{uv} \in A(\vec{H})$ and $S \subseteq V(H)$.

We can now show that $\alpha(H)+2m \leq \alpha(G,w)$. Let $S_H$ be a stable set of $H$ such that $|S_H| = \alpha(H)$. Since $S_H$ is a stable set of $H$, we know that $|S_H \cap \{u,v\}| \leq 1$ for all $\vec{uv} \in A(\vec{H})$. Now, let $Y = \{x_{\vec{uv}} \mid \vec{uv} \in A(\vec{H}), u \in S_H\}$, and set $S_G = S_H \cup Y$. By construction, for all $\vec{uv} \in A(\vec{H})$, we have that $S_G \cap \{u,x_{\vec{uv}},v\} \in \{\{u,x_{\vec{uv}}\},\{v\},\emptyset\}$, and consequently, ${\rm cont}(\vec{uv};S_G) = 2$. Since $|A(\vec{H})| = m$, and since $S_G$ is a stable set of $G$ (because $G$ contains no strongly adjacent pairs), it now follows that
\begin{displaymath}
\begin{array}{rcl}
\alpha(H)+2m & = & |S_H|+\sum\limits_{\vec{uv} \in A(\vec{H})} {\rm cont}(\vec{uv};S_G)
\\
\\
& = & \llbracket S_G \rrbracket_{(G,w)}
\\
\\
& \leq & \alpha(G,w).
\end{array}
\end{displaymath}

It remains to show that $\alpha(G,w) \leq \alpha(H)+2m$. Recall that all subsets of $V(G)$ are stable sets of $G$. Now, among all subsets $S_G$ of $V(G)$ that satisfy $\llbracket S_G \rrbracket_{(G,w)} = \alpha(G,w)$, choose one for which the size of the set $\{\vec{uv} \in A(\vec{H}) \mid u,v \in S_G\}$ is as small as possible. We need to show that $\llbracket S_G \rrbracket_{(G,w)} \leq \alpha(H)+2m$. Since ${\rm cont}(\vec{uv};S) \leq 2$ for all $\vec{uv} \in A(\vec{H})$, and since $|A(\vec{H})| = m$, we see that $\sum\limits_{\vec{uv} \in A(\vec{H})} {\rm cont}(\vec{uv};S_G) \leq 2m$, and consequently,
\begin{displaymath}
\begin{array}{rcl}
\llbracket S_G \rrbracket_{(G,w)} & = & |S_G \cap V(H)|+\sum\limits_{\vec{uv} \in A(\vec{H})} {\rm cont}(\vec{uv};S_G)
\\
\\
& \leq & |S_G \cap V(H)|+2m.
\end{array}
\end{displaymath}
Thus, it only remains to show that $|S_G \cap V(H)| \leq \alpha(H)$. To prove this, we need only show that $S_G \cap V(H)$ is a stable set of $H$. Suppose otherwise, and choose an arc $\vec{u_0v_0} \in A(\vec{H})$ such that $u_0,v_0 \in S_G$. Our goal is to construct a set $S_G' \subseteq V(G)$ such that $\llbracket S_G' \rrbracket_{(G,w)} = \alpha(G,w)$ and such that $|\{\vec{uv} \in A(\vec{H}) \mid u,v \in S_G'\}| < |\{\vec{uv} \in A(\vec{H}) \mid u,v \in S_G\}|$. This will contradict the minimality of $S_G$, which is all we need.

Let $S_G'$ be the subset of $V(G)$ defined as follows:
\begin{itemize}
\item $S_G' \cap V(H) = (S_G \cap V(H)) \smallsetminus \{u_0\}$;
\item for all $\vec{uv} \in A(\vec{H})$,
\begin{itemize}
\item if $u_0 \notin \{u,v\}$, then we set $x_{\vec{uv}} \in S_G'$ if and only if $x_{\vec{uv}} \in S_G$;
\item if $u_0 = u$, then we set $x_{\vec{u_0v}} \notin S_G'$;
\item if $u_0 = v$, then we set $x_{\vec{u_0v}} \in S_G'$.
\end{itemize}
\end{itemize}

Because of the arc $\vec{u_0v_0}$, we see that $|\{\vec{uv} \in A(\vec{H}) \mid u,v \in S_G'\}| < |\{\vec{uv} \in A(\vec{H}) \mid u,v \in S_G\}|$. In order to verify that $S_G'$ contradicts the minimality of $S_G$, it remains to show that $\llbracket S_G' \rrbracket_{(G,w)} = \alpha(G,w)$. By construction, $|S_G' \cap V(H)| = |S_G \cap V(H)|-1$. Next, for all $\vec{uv} \in A(\vec{H})$ such that $u_0 \in \{u,v\}$, we have that ${\rm cont}(\vec{uv};S_G') = 2$, and consequently, ${\rm cont}(\vec{uv};S_G') \geq {\rm cont}(\vec{uv};S_G)$. Furthermore, ${\rm cont}(\vec{u_0v_0};S_G) = 1$, and so ${\rm cont}(\vec{uv};S_G') = 1+{\rm cont}(\vec{uv};S_G)$. On the other hand, for all $\vec{uv} \in A(\vec{H})$ such that $u_0 \notin \{u,v\}$, we have that ${\rm cont}(\vec{uv};S_G') = {\rm cont}(\vec{uv};S_G)$. Thus,
\begin{displaymath}
\begin{array}{rcl}
\sum\limits_{\vec{uv} \in A(\vec{H})} {\rm cont}(\vec{uv};S_G') & \geq & 1+\sum\limits_{\vec{uv} \in A(\vec{H})} {\rm cont}(\vec{uv};S_G),
\end{array}
\end{displaymath}
and it follows that
\begin{displaymath}
\begin{array}{rcl}
\llbracket S_G' \rrbracket_{(G,w)} & = & |S_G' \cap V(H)|+\sum\limits_{\vec{uv} \in A(\vec{H})} {\rm cont}(\vec{uv};S_G')
\\
\\
& \geq & (|S_G \cap V(H)|-1)+(1+\sum\limits_{\vec{uv} \in A(\vec{H})} {\rm cont}(\vec{uv};S_G))
\\
\\
& = & |S_G \cap V(H)|+\sum\limits_{\vec{uv} \in A(\vec{H})} {\rm cont}(\vec{uv};S_G)
\\
\\
& = & \llbracket S_G \rrbracket_{(G,w)}
\\
\\
& = & \alpha(G,w).
\end{array}
\end{displaymath}
Since $S_G'$ is a stable set of $G$ (because $G$ contains no strongly adjacent pairs), we deduce that $\llbracket S_G' \rrbracket_{(G,w)} = \alpha(G,w)$. Thus, $S_G'$ indeed contradicts the minimality of $S_G$. This completes the argument.
\end{proof}

\section*{Acknowledgment} We would like to thank Fr\'ed\'eric Maffray for his help with the proof of Theorem~\ref{Bip-NP}.

\end{document}